\documentclass[11pt]{article}
\usepackage{amssymb,latexsym,amsmath,graphicx,amsfonts,amsthm,xy,eucal,mathrsfs,fullpage}
\usepackage{amscd,subfigure,color,setspace,cite,mathtools}
\PassOptionsToPackage{dvipsnames,svgnames}{xcolor}
\usepackage{xkeyval,tkz-base}
\usetikzlibrary{arrows,calc}
\usepackage{epstopdf}
\usepackage{epsfig}
\usepackage{tikz,pgfplots}
\pgfplotsset{compat=1.10}
\usepgfplotslibrary{fillbetween}
\usetikzlibrary{patterns}
\usepgflibrary{fpu}
 \usetikzlibrary{calc,shapes,decorations,decorations.text, mindmap,shadings,patterns,matrix,arrows,intersections,automata,backgrounds}
\usepackage{float}
\usetikzlibrary{arrows}
\usetikzlibrary{decorations.pathmorphing}
 \usetikzlibrary{plotmarks}
        \usetikzlibrary{intersections}
        \usetikzlibrary{shadings}
        \usetikzlibrary{calc}
        \usetikzlibrary{patterns}
\newlength{\hatchspread}
\newlength{\hatchthickness}
\newlength{\hatchshift}
\newcommand{\hatchcolor}{}
\tikzset{hatchspread/.code={\setlength{\hatchspread}{#1}},
         hatchthickness/.code={\setlength{\hatchthickness}{#1}},
         hatchshift/.code={\setlength{\hatchshift}{#1}},
         hatchcolor/.code={\renewcommand{\hatchcolor}{#1}}}
\tikzset{hatchspread=3pt,
         hatchthickness=0.4pt,
         hatchshift=1pt,
         hatchcolor=black}
\pgfdeclarepatternformonly[\hatchspread,\hatchthickness,\hatchshift,\hatchcolor]
   {custom north west lines}
   {\pgfqpoint{\dimexpr-2\hatchthickness}{\dimexpr-2\hatchthickness}}
   {\pgfqpoint{\dimexpr\hatchspread+2\hatchthickness}{\dimexpr\hatchspread+2\hatchthickness}}
   {\pgfqpoint{\dimexpr\hatchspread}{\dimexpr\hatchspread}}
   {
    \pgfsetlinewidth{\hatchthickness}
    \pgfpathmoveto{\pgfqpoint{0pt}{\dimexpr\hatchspread+\hatchshift}}
    \pgfpathlineto{\pgfqpoint{\dimexpr\hatchspread+0.15pt+\hatchshift}{-0.15pt}}
    \ifdim \hatchshift > 0pt
      \pgfpathmoveto{\pgfqpoint{0pt}{\hatchshift}}
      \pgfpathlineto{\pgfqpoint{\dimexpr0.15pt+\hatchshift}{-0.15pt}}
    \fi
    \pgfsetstrokecolor{\hatchcolor}
    \pgfusepath{stroke}
   }

\newtheorem{theorem}{Theorem}[section]
\newtheorem{definition}{Definition}[section]
\newtheorem{proposition}{Proposition}[section]
\newtheorem{lemma}{Lemma}[section]
\newtheorem{example}{Example}[section]

\newtheorem{remark}{Remark}[section]
\newtheorem{corollary}{Corollary}[section]

\numberwithin{equation}{section}

\definecolor{Blue}{rgb}{0,0,0.9}
\usepackage{chngcntr}
\counterwithin{figure}{section}

\begin{document}

\title{Extension, embedding and global stability in two dimensional monotone maps}
\author{Ahmad Al-Salman,\ Ziyad AlSharawi\thanks{ zsharawi@aus.edu}\ \ and Sadok Kallel}
\date{\vspace{-4ex}}

\maketitle
\begin{abstract}
We consider the general second order difference equation $x_{n+1}=F(x_n,x_{n-1})$ in which $F$ is continuous and of mixed monotonicity in its arguments. In equations with negative terms, a persistent set can be a proper subset of the positive orthant, which motivates studying global stability with respect to compact invariant domains. In this paper, we assume that $F$ has a semi-convex compact invariant domain, then make an extension of $F$ on a rectangular domain that contains the invariant domain. The extension preserves the continuity and monotonicity of $F.$ Then we  use the embedding technique to embed the dynamical system generated by the extended map into a higher dimensional dynamical system, which we use to characterize the asymptotic dynamics of the original system. Some illustrative examples are given at the end.
\end{abstract}

\section{Introduction}
Second order difference equations $x_{n+1}=F(x_n,x_{n-1})$ in which $F$ is non-decreasing in one argument and non-increasing in the other are known as second order difference equations of mixed monotonicity. Such equations have been investigated in the literature \cite{Ca-La2008,Gr-La2005,Ko-La1993,Ku-La2002,Se2003} when the positive orthant is invariant (forward invariant), and the particular interest is usually given to the stability of equilibrium solutions (the fixed points of the map $F$). This interest stems from the fact that the dynamics of the equilibria can be sufficient to characterize the dynamics of the system. For instance, when $f(x)$ is continuous, decreasing, positive, $f(0)>1$ and $xf(x)$ is increasing, the positive quadrant is an invariant of the equation $x_{n+1}=F(x_n,x_{n-1})=x_nf(x_{n-1})+h,\;h>0,$ and the unique positive equilibrium  is a global attractor \cite{Al2013,Ny2008}. Under various levels of  assumptions on the map $F,$ several efforts have been successful in furnishing results of paramount significance in characterizing the asymptotic behaviour \cite{Ko-La1993,Ku-La-Si1998, Ku-Me2006, Go-Ha1994,Sm2008,Sm2006}. The work of Kulenovic et al. \cite{Ku-La-Si1998,Ku-Me2006} assumes the system $(x,y)=(F(x,y), F(y,x))$ has no nontrivial solution, and relies heavily on the assumption that the map $F$ assumes an invariant box to prove  global stability of the unique equilibrium. A more general approach was given by Gouze and Hadeler \cite{Go-Ha1994} and developed further by Smith \cite{Sm2008,Sm2006}. This approach relies on the notion of embedding the dynamical system generated by $F$ into a larger symmetric monotone dynamical system (the references cited in \cite{Sm2008,Sm2006} give a good account of the development of this approach), then the dynamics of the embedded system can be used to characterize the dynamics of the original system.
\\

In discrete systems with negative coefficients, the positive orthant may not be an invariant domain \cite{Ab-Al-Rh2013}; however, it is quite common to limit the interest to solutions that start in the nonnegative orthant and stay within the nonnegative orthant. Such solutions define what is known as the persistent set. When an invariant domain is found, it may not be a box, which makes the developed theory in the aforementioned literature not enough to address the asymptotic dynamics.
In this paper, we are concerned with the global stability analysis of difference equations
\begin{equation}\label{Eq-General}
x_{n+1} = F(x_n, x_{n-1}),\quad n = 0,1,\ldots,
\end{equation}
where $F$ is continuous and monotonic function in its arguments. In particular, $F$
is of mixed monotonicity and assumes a compact invariant domain, which is not
necessarily a box. The case when $F$ is either increasing in both arguments or
decreasing in the first argument while increasing in the second is covered by the results
of Amleh, Ladas and Camouzis \cite{Am-Ca-La2008,Ca-La2006} and Kulenovi\'{c} and Merino \cite{Ku-Me2006,Ku-Me2009,Ku-Me2010}. In
these cases every solution breaks into two monotone subsequences and so every
bounded solution converges to either an equilibrium solution, period-two
solution or to the singular point on the boundary. Thus in this case the main
problem is to determine the basins of attraction of different attractors. This
is done using the results in \cite{Ku-Me2006,Ku-Me2009,Ku-Me2010} which clearly determine the boundaries of
these basins of attraction as the stable manifolds of some saddle point or nonhyperbolic
equilibrium solutions or period-two solutions.
\\

 For writing and reading conveniences, we use $F(\uparrow,\cdot)$ or $F(\downarrow,\cdot)$ to denote that $F$ is non-decreasing or non-increasing in its first variable, respectively.   Similarly, for the second variable. Thus, our interest in this paper is limited to maps that satisfy  $F(\uparrow,\uparrow), $ $F(\uparrow,\downarrow),$ $F(\downarrow,\uparrow)$ or $F(\downarrow,\downarrow).$
We focus on $F(\uparrow,\downarrow),$ and prove that for a continuous map $z=F(x,y)$ with mixed monotonicity on a compact and connected invariant domain $\Omega,$ if
$ T(x,y)=(F(x,y),x): \Omega\;\rightarrow\;\Omega$
and $F$ can be extended to $\widetilde{F}$ over a rectangular invariant domain containing $\Omega$ such that the system
$(\widetilde{F}(x,y),\widetilde{F}(y,x))=(x,y)$
 has no solutions other than the fixed points of $F,$ then $x_{n+1}=F(x_n,x_{n-1})$ must converge to a fixed point of $F$ for all  $(x_0,x_{-1})\in \Omega.$ We achieve this task based on several results. First, we give a method for constructing an extension in section two. Then in section three, we embed the extended system into a higher dimensional system using several auxiliary maps, then discuss the characteristics of the embedded system and its effect on the asymptotic behavior of the original systems. In section four, we consider a concrete case and apply the developed theory to establish global stability with respect to a certain invariant domain.  Finally, we end this paper by some discussion and conclusion.

\section{Invariant domains and map extension}

This paper considers the second order difference equation
\begin{equation}\label{Eq-DE1}
x_{n+1}=F(x_n,x_{n-1}), \quad \text{where}\quad n=0,1,\ldots \quad \text{and}\quad  x_{-1},x_0\geq 0.
 \end{equation}
where $F$ is continuous on a compact and connected invariant domain $\Omega\subset [0,\infty)\times [0,\infty)=:{\mathbb{R}_+}^2.$ Here $\Omega$ is invariant in the sense that if $(x,y)\in \Omega,$ then $(F(x,y),x)\in \Omega.$ Obviously, $F$ must have at least one fixed point $x^*,$ i.e., $F(x^*,x^*)=x^*,$ which is called an equilibrium solution (or a steady state solution) of Eq. (\ref{Eq-DE1}). Furthermore, we may consider  $F:\; \Omega\;\rightarrow\; [m,M],$ where
 $$m:=\min_{(x,y)\in \Omega} F\quad \text{and}\quad M:=\max_{(x,y)\in  \Omega} F,$$
 and by starting from the second iterate of Eq. (\ref{Eq-DE1}), there is no loss of generality if we assume $\Omega\subseteq [m,M]\times[m,M]=:[m,M]^2.$
\\

 Throughout this paper, we consider two partial order relations on $\Omega,$ namely, the southeast order $\leq_{se}$ and the northeast order $\leq_{ne}.$ The southeast order is defined as
$$
(x,y)\leq_{se}(u,v)\quad \text{if and only if}\quad x\leq u\quad \text{and}\quad v\leq y,
$$
while the northeast order is defined as
$$
(x,y)\leq_{ne}(u,v)\quad \text{if and only if}\quad x\leq u\quad \text{and}\quad y\leq v.
$$
A well known observation \cite{Go-Ha1994, Sm2006,Sm2008}
is that if a function $F:\, I^2\rightarrow I$, with $I=[a,b]$ or $I=\mathbb{R}_+$, satisfies $F(\uparrow,\downarrow),$ then the symmetric map
\begin{equation}
G(x,y)=(F(x,y),F(y,x))
\label{Eq-SymmetricMap1}
\end{equation}%
is non-decreasing with respect to $\leq_{se},$ i.e., $G(\uparrow)$. This simple observation has had useful applications in the literature \cite{Go-Ha1994,Sm2006,Sm2008}
for it allows to deduce global convergence
properties for $F$ from the dynamics of the symmetric map $G$ in favorable situations (see \S\ref{embeddingsection}).
In general however, given a map $F:\; \Omega\rightarrow [m,M]$, $\Omega$ may not necessarily be a rectangular region, and the map $G$ is not immediately well-defined.
(see the examples of \S\ref{Sec-Application}). This phenomenon makes
 the above definition of the symmetric map $G$ ill-defined. In this section, we go around this problem by showing that we can extend a continuous map $F: \;\Omega\;\rightarrow [a,b]$, $\Omega\subset [a,b]^2$ under some mild conditions into a map $\widetilde{F}:\; [a,b]^2\;\rightarrow\; [a,b]$ that satisfies the following properties:
 \begin{description}
 \item{(C1)} $\widetilde{F}$ is continuous on $[a,b]^2\supseteq \Omega.$
 \item{(C2)} $F(x,y)=\widetilde{F}(x,y)$ for all $(x,y)\in \Omega.$
 \item{(C3)} $\widetilde{F}$ has the same monotonicity behaviour as $F.$ In particular, $\widetilde{F}(\uparrow,\downarrow)$ whenever $F(\uparrow,\downarrow),$ and $\widetilde{F}(\downarrow,\uparrow)$ whenever $F(\downarrow,\uparrow).$
 \end{description}
 Throughout this paper, whenever talking about an extension $\widetilde{F},$ $\widetilde{F}$ must satisfy (C1), (C2) and (C3). It is interesting to point out that
the image set of our extension $\widetilde{F}$ is the same as that of $F$. We formalize these concepts in the following definition:

\begin{definition}\label{defext}
Let $F$ be a continuous function of mixed monotonicity on a domain $\Omega\subset \mathbb{R}^2.$ A continuous function $\widetilde{F}$ is called \emph{an extension} of $F$ if $\widetilde{F}(X)=F(X)$ for all $X\in \Omega$ and $\widetilde{F}$ is of same mixed monotonicity as $F.$  Further, $\widetilde{F}$ is called a \emph{nice extension} if  the range of $\widetilde{F}$ and the range of $F$ coincide.
 \end{definition}

 \subsection{Extending $F$ from a convex domain to a rectangular domain}
 Assume $\Omega$ is compact, and we start by assuming the boundary $\partial \Omega$ to be a piecewise smooth Jordan curve (simple closed). We find it convenient to focus on the extension of $F(\downarrow,\uparrow).$   The extension of $G(\uparrow,\downarrow)$ can be obtained by considering $\widetilde{G}(x,y)=\widetilde{F}(y,x).$ For a point $p=(x,y)$ in the plane, we write
$(x^+,y)$ (respectively $(x^-,y)$) the point obtained as the intersection of the horizontal line through $p$ with the curve $\partial\Omega$ on the right (respectively on the left), if this point exists. Similarly we write $(x,y^+)$ (respectively $(x,y^-)$) the point obtained as the intersection of the vertical line through $p$ with the curve $\partial\Omega$ on the top (respectively on the bottom). Part (a) of Fig. \ref{Fig-Origami1} illustrates our notations and the projection of an external point to the boundary $\partial \Omega.$
\begin{center}
\begin{figure}[!h]
\definecolor{ffqqqq}{rgb}{1.,0.,0.}
\begin{minipage}[t]{0.48\textwidth}
\raggedleft
\begin{center}
\begin{tikzpicture}[scale=0.7]
\draw[line width=0.8pt,color=blue,pattern=dots,pattern color=brown]  plot[smooth, tension=.8] coordinates {(-3.5,0.5) (-3,2.0) (-1,3.5) (1.5,3) (3,2.5) (4.5,2.0) (4,1.0) (2.5,0) (1,-1.0) (-1,-2) (-2.5,-1) (-3.5,0.5)};
\draw[line width=1.0pt,scale=1.0,red] (-1,-2.05) -- (-3.53,-2.0) -- (-3.53,3.6) -- (4.6,3.6) -- (4.6,-2.0) -- (-1.0,-2.05);
\draw[xscale=1.0] (-0.5,1.0) node[right,rotate=0] {\Large $\Omega$};
\draw[xscale=1.0] (3.0,3.0) node[above,rotate=0] { \footnotesize $(x,y)$};
\draw[line width=1.0pt,red,fill=yellow] (3.0,3.0) circle (1.5pt);
\draw[line width=0.6pt,scale=1.0,black,dotted] (3,3) -- (3,2.4);
\draw[xscale=1.0] (3.0,2.55) node[below,rotate=-10] { \footnotesize $\boldsymbol{(x,y^-)}$};
\draw[line width=1.0pt,red,fill=yellow] (3.0,2.5) circle (1.5pt);
\draw[line width=0.6pt,scale=1.0,black,dotted] (3,3) -- (1.7,3);
\draw[xscale=1.0] (1.5,3.05) node[below,rotate=-10] { \footnotesize $\boldsymbol{(x^-,y)}$};
\draw[line width=1.0pt,red,fill=yellow] (1.7,3.0) circle (1.5pt);
\draw[xscale=1.0] (3.0,-1.0) node[right,rotate=0] {\footnotesize $(x,y)$};
\draw[line width=1.0pt,red,fill=yellow] (3.0,-1.0) circle (1.5pt);
\draw[line width=0.6pt,scale=1.0,black,dotted] (3,-1) -- (3,0.4);
\draw[xscale=1.0] (3.0,0.5) node[below,rotate=40] { \footnotesize $(x,y^+)$};
\draw[line width=1.0pt,red,fill=yellow] (3.0,0.35) circle (1.5pt);
\draw[line width=0.6pt,scale=1.0,black,dotted] (3,-1) -- (1,-1);
\draw[xscale=1.0] (1.0,-1.0) node[below,rotate=30] { \footnotesize $(x^-,y)$};
\draw[line width=1.0pt,red,fill=yellow] (1.0,-1.0) circle (1.5pt);
\draw[xscale=1.0] (-3.5,-1.7) node[right,rotate=0] {\footnotesize $\boldsymbol{(x,y)}$};
\draw[line width=1.0pt,red,fill=yellow] (-3.0,-1.5) circle (1.5pt);
\draw[line width=0.6pt,scale=1.0,black,dotted] (-3.0,-1.5) -- (-2.0,-1.5);
\draw[xscale=1.0] (-2.2,-1.0) node[right,rotate=-35] { \footnotesize $\boldsymbol{(x^+,y)}$};
\draw[line width=1.0pt,red,fill=yellow] (-2.0,-1.5) circle (1.5pt);
\draw[line width=0.6pt,scale=1.0,black,dotted] (-3,-1.5) -- (-3,-0.2);
\draw[xscale=1.0] (-3.5,0.0) node[above right,rotate=-45] { \footnotesize $\boldsymbol{(x,y^+)}$};
\draw[line width=1.0pt,red,fill=yellow] (-3.0,-0.5) circle (1.5pt);
\draw[xscale=1.0] (-3.0,3.0) node[above,rotate=0] {\footnotesize $(x,y)$};
\draw[line width=1.0pt,red,fill=yellow] (-3.0,3.0) circle (1.5pt);
\draw[line width=0.6pt,scale=1.0,black,dotted] (-3.0,3.0) -- (-2.0,3.0);
\draw[line width=1.0pt,red,fill=yellow] (-2.0,3.0) circle (1.5pt);
\draw[line width=0.6pt,scale=1.0,black,dotted] (-3,3.0) -- (-3,2.0);
\draw[line width=1.0pt,red,fill=yellow] (-3.0,2.0) circle (1.5pt);
\draw[xscale=1.0] (0.0,-2.5) node[below] {\footnotesize (a)};
\end{tikzpicture}
\end{center}
\end{minipage}
\begin{minipage}[t]{0.48\textwidth}
\raggedright
\begin{center}
\begin{tikzpicture}[scale=0.7]
\draw[line width=1.2pt,color=blue,dotted,pattern=dots,pattern color=brown]  plot[smooth, tension=.8] coordinates {(-3.5,0.5) (-3,2.0) (-1,3.5) (1.5,3) (3,2.5) (4.5,2.0) (4,1.0) (2.5,0) (1,-1.0) (-1,-2) (-2.5,-1) (-3.5,0.5)};
\draw[line width=0.8pt,color=blue,pattern=dots,pattern color=brown] (-3.5,0.5) -- (-3.5,2.0) -- (-3,2.0) -- (-3,2.5) -- (-2.5,2.5) --(-2.5,3) -- (-2.5,3.55)-- (1.5,3.55) --(1.5,3) -- (2.5,3) -- (2.5,2.7) -- (4,2.7) -- (4,2.25) -- (4.55,2.25) -- (4.55,1) -- (4,1.0) --(4,0) --(2.5,0)--(2.5,-1)-- (2.0,-1) -- (1,-1.0)--(1,-1.98) -- (-2,-1.98) -- (-2,-1.5)--(-2.8,-1.5)--(-2.8,-0.75)--(-3.48,-0.75) -- (-3.48,0.5);
\draw[line width=1.0pt,scale=1.0,red] (-1,-2.05) -- (-3.53,-2.0) -- (-3.53,3.6) -- (4.6,3.6) -- (4.6,-2.0) -- (-1.0,-2.05);
\draw[xscale=1.0] (-0.5,1.0) node[right,rotate=0] {\Large $\Omega$};
\draw[xscale=1.0] (0.0,-2.5) node[below] {\footnotesize (b)};
\end{tikzpicture}
\end{center}
\end{minipage}%
\caption{In part (a) of this figure, we illustrate our notation of projecting a point $(x,y)$ to the boundary $\partial \Omega.$ In part (b), we illustrate the notion of putting the invariant domain $\Omega$ inside an Origami domain. }\label{Fig-Origami1}
\end{figure}
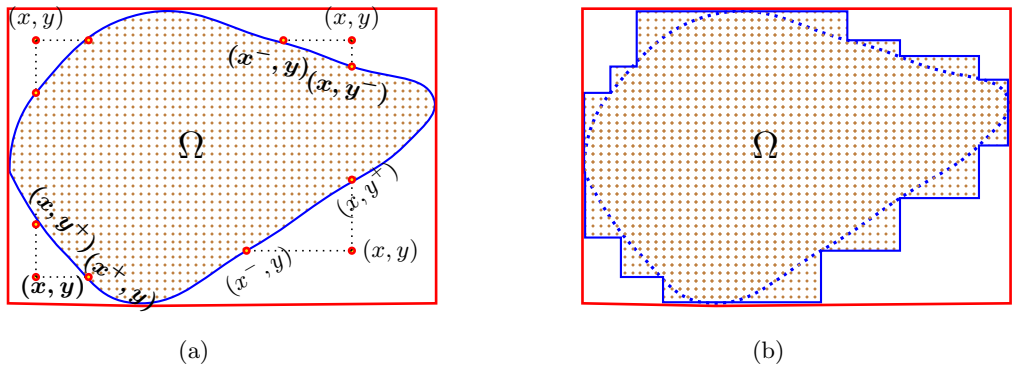
\end{center}

The next result gives a condition and a property on $F$ which are needed in the sequel.

\begin{lemma}\label{Lem-Sadok1}
Let $\Omega \subseteq \mathbb{R}^{2}$ be compact such that the
boundary $\partial \Omega $ forms a Jordan curve. Suppose that $F$ is
continuous on $\Omega .$ If ${\bf r} :[0,1]\longrightarrow \partial \Omega $ is a
parametrization so that $f:=F\circ {\bf r} $ is differentiable on $[0,1]$
and $f^{\prime }\;$ has only finitely many zeroes$,$ then $\partial\Omega$
can be decomposed into finite number of pieces for which $f$ is monotonic.
\end{lemma}

\begin{proof} Let $0\leq t_{1}<t_{2}<...<t_{m}\leq 1$ be the
zeroes of $f^{\prime }$. Then $f^{\prime }$ has a single
sign in each interval $(t_{i},t_{i+1}),1\leq i\leq m-1$. Otherwise,
if $f^{\prime }$ changes sign in $(t_{i},t_{i+1})$ for
some $i$, then by Dini's theorem, $f^{\prime }$  must have a
zero in $(t_{i},t_{i+1})$ which is impossible. Since $\partial
\Omega $ forms a Jordan curve, there exists a continuous map ${\bf r}
:\;[0,1]\;\rightarrow \;\mathbb{R}^{2}$ such that ${\bf r} (0)={\bf r} (1)$
and the restriction of ${\bf r} $ to $[0,1)$ is injective. From the
assumption that $f$ is differentiable on $[0,1]$ and $f^{\prime }\;$%
has only finitely many zeroes$,$ for each $t\in \lbrack 0,1],$
there exists an open interval $I_{t}\subseteq \lbrack 0,1]$ such that $%
f^{\prime }(t)\geq 0$ or $f^{\prime }(t)<0.$ Now, $\displaystyle{\cup
_{t}I_{t}}$ forms an open cover of $[0,1].$ The compactness implies the
existence of a finite subcover, say $I_{0},I_{1},\ldots ,I_{n}.$ Hence, $f$
is monotonic on each piece of the finite subcover, which completes the proof.
\end{proof}

Consider a portion of $\partial \Omega$ parametrized by ${\bf r}(t)=(x(t),y(t)),  t\in I\subset [0,1],$ for which $F\circ {\bf r}:\;I\rightarrow\mathbb R$ is monotonic. We explain how to extend $F$ on a rectangular domain containing that portion.

\begin{lemma}[Ray Extension]\label{Lem-Sadok2}
Let $F$ be continuous and of mixed monotonicity on a compact and convex set $\Omega\subset \mathbb{R}^2.$ Suppose that  ${\bf r}(t),\; t_0\leq t\leq t_1$ is a counterclockwise parametrization of a portion of $\partial \Omega$ on which $f=F\circ{\bf r}$ is monotonic. Then $F$ can be extended to a function of same mixed monotonicity on the portion of a rectangular domain bounded on the left by $\bf r$.
\end{lemma}

\begin{proof}
Again our proof is given for $F(\downarrow,\uparrow).$ If ${\bf r} (t)$ is of the form $(x(t),c)$ (horizontal) or $(c,y(t))$ (vertical) on $[t_0,t_1]$, then nothing needs to be done. Otherwise we proceed in cases.
 If $f(\uparrow),$ we have three cases to tackle as shown in Part (a) of Fig.  \ref{Fig-Convex1}. Part (b) of Fig.  \ref{Fig-Convex1} covers the case $f(\downarrow)$. For a point $p=(x,y)$ out of the shaded region, the arrow in each case indicates the direction of the projection we take. For instance, if $p$ is in the unshaded region of (iv) in Part (a), then
$$\widetilde{F}(x,y) = \begin{cases} F(x,y),&\text{if}\; (x,y)\in\Omega\\
F({x^-},y),& \text{if}\; (x,y)\in \Omega^c,
\end{cases}$$
and if $p$ is in the unshaded region of (ii) in Part (a), then
$$\widetilde{F}(x,y) = \begin{cases} F(x,y),&\text{if}\; (x,y)\in\Omega\\
F({x},y^-),&\text{if}\; (x,y)\in \Omega^c.
\end{cases}$$
The other cases follow the same principle.
\end{proof}

\begin{center}
\begin{figure}[!h]
\definecolor{zzttqq}{rgb}{0.6,0.2,0.}
\definecolor{cqcqcq}{rgb}{0.7529411764705882,0.7529411764705882,0.7529411764705882}
\begin{minipage}[t]{0.5\textwidth}
\raggedleft
\begin{center}
\begin{tikzpicture}[line cap=round,line join=round,x=1.0cm,y=1.0cm,scale=1.2]
\draw[-triangle 45,line width=1.0pt,smooth,blue] (1.3,1.3) -- (1.1,1.5);
\draw[-triangle 45,line width=1.0pt,smooth,blue] (-1.3,1.3) -- (-1.5,1.1);
\draw[line width=0.8pt,smooth,green,->] (1.7,1.6) -- (1.15,1.6);
\draw[line width=0.8pt,smooth,green,->] (-1.6,1.7) -- (-1.6,1.15);
\draw[line width=0.8pt,smooth,green,->] (1.7,-1.6) -- (1.15,-1.6);
\draw[color=black] (1,1.7) node[above] {\footnotesize (i)};
\draw[color=black] (-1,1.7) node[above] {\footnotesize (ii)};
\draw[color=black] (1,-0.35) node[above] {\footnotesize (iv)};
\draw[-triangle 45, line width=1.0pt,smooth,blue] (1.3,-1.3) -- (1.5,-1.1);
\draw[line width=0.8pt,smooth,red] (1.75,0.25) -- (1.75,1.75)--(0.25,1.75);
\draw[line width=0.8pt,smooth,red] (-0.25,1.75) -- (-1.75,1.75)--(-1.75,0.25);
\draw[line width=0.8pt,xscale=1.0,smooth,red] (0.25,-1.75) -- (1.75,-1.75)--(1.75,-0.25);
\draw [shift={(0.25,0.25)},line width=0.8pt,color=blue,fill=zzttqq,fill opacity=0.10000000149011612]  (0,0) --  plot[domain=0.:1.5707963267948966,variable=\t]({1.*1.5*cos(\t r)+0.*1.5*sin(\t r)},{0.*1.5*cos(\t r)+1.*1.5*sin(\t r)}) -- cycle ;
\draw [shift={(-0.25,0.25)},line width=0.8pt,color=blue,fill=zzttqq,fill opacity=0.10000000149011612]  (0,0) --  plot[domain=1.5707963267948966:3.141592653589793,variable=\t]({1.*1.5*cos(\t r)+0.*1.5*sin(\t r)},{0.*1.5*cos(\t r)+1.*1.5*sin(\t r)}) -- cycle ;
\draw [shift={(0.25,-0.25)},line width=0.8pt,color=blue,fill=zzttqq,fill opacity=0.10000000149011612]  (0,0) --  plot[domain=-1.5707963267948966:0.,variable=\t]({1.*1.5*cos(\t r)+0.*1.5*sin(\t r)},{0.*1.5*cos(\t r)+1.*1.5*sin(\t r)}) -- cycle ;
\draw[xscale=1.0] (0.9,0.5) node[above,rotate=0] {\large $\Omega$};
\draw[xscale=1.0] (-0.9,0.5) node[above,rotate=0] {\large $\Omega$};
\draw[xscale=1.0] (0.9,-0.5) node[below,rotate=0] {\large $\Omega$};
\draw[xscale=1.0] (0.0,-2.0) node[below] {\footnotesize (a) $f(\uparrow)$};
\end{tikzpicture}
\end{center}
\end{minipage}
\begin{minipage}[t]{0.5\textwidth}
\raggedright
\begin{center}
\begin{tikzpicture}[line cap=round,line join=round,x=1.0cm,y=1.0cm,scale=1.2]
\draw[-triangle 45,line width=1.0pt,smooth,blue] (-1.3,-1.3) -- (-1.1,-1.5);
\draw[-triangle 45,line width=1.0pt,smooth,blue] (-1.3,1.3) -- (-1.5,1.1);
\draw[line width=0.8pt,smooth,green,->] (-1.7,1.6) -- (-1.15,1.6);
\draw[line width=0.8pt,smooth,green,->] (-1.7,-1.6) -- (-1.15,-1.6);
\draw[line width=0.8pt,smooth,green,->] (1.6,-1.7) -- (1.6,-1.15);
\draw[color=black] (-1,-0.35) node[above] {\footnotesize (iii)};
\draw[color=black] (-1,1.7) node[above] {\footnotesize (ii)};
\draw[color=black] (1,-0.35) node[above] {\footnotesize (iv)};
\draw[-triangle 45, line width=1.0pt,smooth,blue] (1.3,-1.3) -- (1.5,-1.1);
\draw[line width=0.8pt,smooth,red] (-0.25,1.75) -- (-1.75,1.75)--(-1.75,0.25);
\draw[line width=0.8pt,xscale=1.0,smooth,red] (-1.75,-0.25) -- (-1.75,-1.75)--(-0.25,-1.75);
\draw[line width=0.8pt,xscale=1.0,smooth,red] (0.25,-1.75) -- (1.75,-1.75)--(1.75,-0.25);
\draw [shift={(-0.25,0.25)},line width=0.8pt,color=blue,fill=zzttqq,fill opacity=0.10000000149011612]  (0,0) --  plot[domain=1.5707963267948966:3.141592653589793,variable=\t]({1.*1.5*cos(\t r)+0.*1.5*sin(\t r)},{0.*1.5*cos(\t r)+1.*1.5*sin(\t r)}) -- cycle ;
\draw [shift={(-0.25,-0.25)},line width=0.8pt,color=blue,fill=zzttqq,fill opacity=0.10000000149011612]  (0,0) --  plot[domain=3.141592653589793:4.71238898038469,variable=\t]({1.*1.5*cos(\t r)+0.*1.5*sin(\t r)},{0.*1.5*cos(\t r)+1.*1.5*sin(\t r)}) -- cycle ;
\draw [shift={(0.25,-0.25)},line width=0.8pt,color=blue,fill=zzttqq,fill opacity=0.10000000149011612]  (0,0) --  plot[domain=-1.5707963267948966:0.,variable=\t]({1.*1.5*cos(\t r)+0.*1.5*sin(\t r)},{0.*1.5*cos(\t r)+1.*1.5*sin(\t r)}) -- cycle ;
\draw[xscale=1.0] (-0.9,-0.5) node[below,rotate=0] {\large $\Omega$};
\draw[xscale=1.0] (-0.9,0.5) node[above,rotate=0] {\large $\Omega$};
\draw[xscale=1.0] (0.9,-0.5) node[below,rotate=0] {\large $\Omega$};
\draw[xscale=1.0] (0.0,-2.0) node[below] {\footnotesize (b) $f(\downarrow)$};
\end{tikzpicture}
\end{center}
\end{minipage}%
\caption{This figure shows the possible options for the boundary of a convex $\Omega$ and a possible ray extension.  Part (a) is based on the assumption that $f$ is non-decreasing and Part (b) is based on the assumption that $f$ is non-increasing. Note that the missing quarter in Part (a) is due to the fact that we cannot have $f(\uparrow)$ and $F(\downarrow,\uparrow)$ at the same time. Similarly for the missing quarter of Part (b). }\label{Fig-Convex1}
\end{figure}
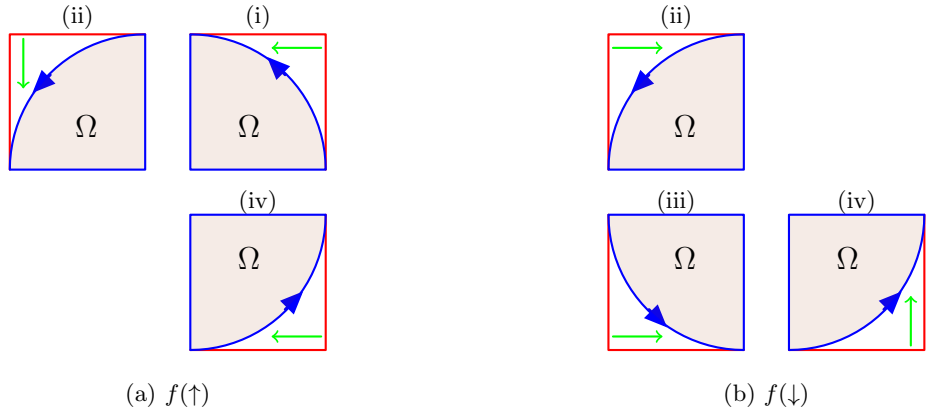
\end{center}

Extensions as in Lemma \ref{Lem-Sadok2} above are dubbed ``ray extensions" since $\widetilde{F}$ is always constant on either horizontal or vertical rays in $\Omega^c$.
Ray extensions allow us in turn to extend our map $F$ from convex $\Omega$ to a polygonal domain whose edges are horizontal or vertical, and each edge must intersect $\partial\Omega$. We call such a domain an ``origami domain". Part (b) of Fig. \ref{Fig-Origami1} illustrates such a domain containing $\Omega.$
From now on we will have to assume that the boundary of $\Omega$ is parameterized as a piecewise regular curve\footnote{A \textit{regular} curve is a curve whose derivative at every point exists and is non-zero.} ${\bf r} : [0,1]\longrightarrow\mathbb R^2$ that is counterclockwise oriented and such that the derivative of $f=F\circ{\bf r}$ has finitely many zeros, so that by Lemma \ref{Lem-Sadok1}, $F$ restricted to $\partial\Omega$ is monotonic on a finite number of pieces.

\begin{lemma}\label{Lem-Origami1}
Let $\Omega$ be a compact and convex region in the plane of which boundary is a piecewise smooth Jordan curve, and let $F:\; \Omega\rightarrow \mathbb{R}$ be a continuous mixed monotonic function on $\Omega$ as in Lemma \ref{Lem-Sadok1}. Then $F$ extends to a function $\widetilde{F}$ on an origami domain. Furthermore, this extension can be chosen so that its image coincides with that of $F$.
\end{lemma}

\begin{proof}
Our curve $\partial\Omega$ is oriented counterclockwise bounding a convex region $\Omega$. It has a parameterization ${\bf r}: [0,1]\rightarrow\partial\Omega$, ${\bf r} (0)={\bf r} (1)$. For a connected portion $C_j$ of the curve $C$, we will write its restricted parameterization as ${\bf r}_j: I_j\subset [0,1]\rightarrow C_j$, where ${\bf r}_j={\bf r}_{|I_j}$.

By assumption (see Lemma \ref{Lem-Sadok1}), the curve splits into a finite number of segments $C_1,C_2,\ldots, C_n$ such that (i) these segments are either entirely horizontal or vertical, or of the six forms given in Figure \ref{Fig-Convex1}, and (ii)
each restriction $f_j := F\circ{\bf r}_j$ is monotonic (increasing or decreasing) from
$I_j=[a_j,b_j]\subset [0,1]$ into $\mathbb{R}$.

For simplicity, we say that we ``fill in" a region if we are able to extend $F$ over that region. If $C_j$ is entirely vertical or horizontal, then nothing needs to be done. Otherwise,we extend
$F$ to the right of the curve (here ``right" is with respect to the orientation which is counterclockwise) by ray extension.  The extension fills in the corresponding box depicted in Figure \ref{Fig-Convex1}. The region to be filled has horizontal top and bottom sides, left boundary $C_j$ (again ``left" w.r.t the counterclockwise orientation) and right most boundary the line segment having vertex
$({\bf r}_j(a_j)_x, {\bf r}_j(b_j)_y)$ or $({\bf r}_j(a_j)_y,{\bf r}_j(b_j)_x)$ depending on the situation
(we use the notation that if $p\in\mathbb R^2$, then $p$ has coordinates $(p_x,p_y)$).
We now extend $F$ for the six cases by taking the appropriate ray extension as described in Lemma \ref{Lem-Sadok2} and depicted in Figure \ref{Fig-Convex1}.
Going around our boundary curve $\partial\Omega$, with $\Omega$ convex, and extending $F$ to $\widetilde{F}$ on those boxes to the right of the $C_j$'s, we obtain an extension of $F$ to a larger domain that is an origami domain.
\end{proof}

The next step is to extend from the origami domain into a rectangular domain. To that end, we will define the notion of ``grafting". Let $\gamma_i: I_i\rightarrow \mathbb{R}^3$, $i=1,2$, be two curves in space. The following process explains how to ``graft" the first curve $\gamma_1$ to a point of $\gamma_2$. We write $p_0 = \gamma_1(0)$ the starting point of $\gamma_1$. For a point $p=\gamma_2(t)$ on the second curve, we
consider the new curve
$$\gamma_{p} : I_1\rightarrow \mathbb{R}^3\ \ \ ,\ \ \gamma_p(s) = \gamma_1(s) + \overrightarrow{p_0p}$$
This is the curve ${\bf r}_1$ translated along the vector $\overrightarrow{p_0p}$ so it starts at the point
$\gamma_2(t)$; i.e.
$\gamma_p(0)  = \gamma_2(t)$. We say that ``we graft $\gamma_1$ along $\gamma_2$" means that we graft $\gamma_1$ at every point of $\gamma_2$, obtaining a new replica of $\gamma_1$ starting at each point of $\gamma_2$.
For suitable choices of $\gamma_1$ and $\gamma_2$, this process yields a surface in space as illustrated in Figure \ref{Fig-Grafting}.  The operation of grafting is not commutative.
\\

The following example is easy to check and is what we use in the sequel.

\begin{example}  Let $\gamma_1$ and $\gamma_2$ be defined by
$\gamma_1(s) = (x_0, s, z(s))$
and
$\gamma_2(t) = (t, y_0, w(t))$, with $s,t\in [0,1]$.
Grafting $\gamma_1$ along $\gamma_2$ yields the parameterized surface $\nu : [0,1]^2\rightarrow \mathbb{R}^3$ defined by
$$\nu (t,s) = (t, y_0+s, z(s) + w(t) - z(0)).$$
In particular, if $z$ and $w$ are monotonous, then $\nu$ is mixed monotonous. If the surface is seen as a graph of $F$ in the variables $(x=t,y=y_0+s)$, then $F(\uparrow,\downarrow)$ if $z\uparrow$ and $w\downarrow.$
\end{example}

\begin{center}
\begin{figure}[!h]
\definecolor{zzttqq}{rgb}{0.6,0.2,0.}
\definecolor{cqcqcq}{rgb}{0.7529411764705882,0.7529411764705882,0.7529411764705882}
\begin{minipage}[t]{0.5\textwidth}
\begin{center}
\begin{tikzpicture}[line cap=round,line join=round,x=1.0cm,y=1.0cm,scale=1.0]
\fill[Maroon!20,opacity=.5] (1,2.5,0)-- (3,2.5,0) -- (3,2,2)--(1,2,2) -- (1,2.5,0);
\draw[line width=0.7pt,dotted,red] (1,0,0) -- (1,2.5,0);
\draw[line width=0.7pt,dotted,red] (1,0,2)-- (1,2,2);
\draw[line width=0.7pt,dotted,red] (3,0,2)-- (3,2,2);
\draw[line width=0.7pt,dotted,red] (3,0,0)-- (3,2.5,0);
\draw[line width=0.8pt,color=blue] (1,0,0) -- (1.,0,2)--(3,0,2)--(3,0,0)--(1,0,0);
\draw [-triangle 45,line width=0.8pt,color=blue] (0.9,2.4,-0.1) to[out=10,in=170] (2.9,2.4,-0.1);
\draw [->,shift={(0,0.1,-0.3)},line width=0.6pt,color=blue,dashed] (1.1,2,2) to[out=10,in=150] (3,2,2);
\draw [->,shift={(0,0.2,-0.6)},line width=0.6pt,color=blue,dashed] (1.1,2,2) to[out=10,in=150] (3.0,2,2);
\draw [->,shift={(0,0.3,-0.9)},line width=0.6pt,color=blue,dashed] (1.1,2,2) to[out=10,in=150] (3,2,2);
\draw [->,shift={(0,0.4,-1.2)},line width=0.6pt,color=blue,dashed] (1.1,2,2) to[out=10,in=150] (3,2,2);
\draw [->,shift={(0,0.5,-1.5)},line width=0.6pt,color=blue,dashed] (1.1,2,2) to[out=10,in=160] (3,2,2);
\draw [->,shift={(0,0.0,-0.1)},line width=0.6pt,color=blue,dashed] (1.1,2,2) to[out=10,in=160] (3,2,2);
\draw[-triangle 45,line width=0.8pt,smooth,blue] (1,0,0.5) -- (1,0,1.0);
\draw[-triangle 45,line width=0.8pt,smooth,blue] (1.5,0,0) -- (2,0,0);
\draw[-triangle 45,line width=0.8pt,smooth,blue] (1,2.5,0) -- (1,2,2.0);
\draw[-triangle 45, line width=1.0pt,color=black,dashed] (0,0,-0.5) -- (3.5,0.,-0.5) node[below right] {$y$};
\draw[-triangle 45, line width=1.0pt,color=black,dashed] (0,0,-0.5) -- (0.,3.5,-0.5) node[below right] {$z$};
\draw[-triangle 45, line width=1.0pt,color=black,dashed] (0,0,-0.5) -- (0.,0.,4.0) node[below right] {$x$};
\draw[line width=1.0pt,color=black] ((2.2,2.9,0.0) node[below] {\footnotesize $\gamma_1$};
\draw[line width=1.0pt,color=black] ((1,2.2,0.5) node[left] {\footnotesize $\gamma_2$};
\draw[line width=1.0pt,color=black] ((0,0,2.0) node[left] {\footnotesize $x_0$};
\draw[line width=1.0pt,color=black] ((1,0.4,0.3) node[right] {\footnotesize $y_0$};
\draw[line width=0.7pt,dashed,black] (1,0,2)-- (0,0,2);
\end{tikzpicture}
\end{center}
\end{minipage}
\begin{minipage}[t]{0.5\textwidth}
\begin{center}
\begin{tikzpicture}[line cap=round,line join=round,x=1.0cm,y=1.0cm,scale=1.0]
\fill[Maroon!20,opacity=.5] (1,0,1) -- (1.,0,3)--(3,0,3)--(3,0,1)--(1,0,1);
\draw[line width=0.8pt,color=blue] (1,0,1) -- (1.,0,3)--(3,0,3)--(3,0,1)--(1,0,1);
\draw[-triangle 45,line width=0.8pt,smooth,blue] (1,0,1.0) -- (1,0,3.0);
\draw[-triangle 45,line width=0.8pt,smooth,blue] (2.5,0,1) -- (3,0,1);
\draw[-triangle 45, line width=1.0pt,color=black,dashed] (0,0,0) -- (3.5,0.,0.0) node[below right] {$y$};
\draw[-triangle 45, line width=1.0pt,color=black,dashed] (0,0,0) -- (0.,3.5,0.0) node[below right] {$z$};
\draw[-triangle 45, line width=1.0pt,color=black,dashed] (0,0,0) -- (0.,0.,4.0) node[below right] {$x$};
\draw[line width=1.0pt,color=black] ((1.2,0,3.2) node[left] {\footnotesize $D$};
\draw[line width=1.0pt,color=black] ((3.0,0,1.2) node[right] {\footnotesize $B$};
\draw[line width=1.0pt,color=black] ((1.0,0,0.8) node[left] {\footnotesize $A$};
\draw[line width=1.0pt,color=black] ((3.5,0,3.2) node[left] {\footnotesize $C$};
\draw[line width=1.0pt,color=black] ((0,0,3.0) node[left] {\footnotesize $a_1$};
\draw[line width=1.0pt,color=black] ((0.8,0,-0.3) node[right] {\footnotesize $a_2$};
\draw[line width=1.0pt,color=black] ((2,0,1.35) node[right] {\footnotesize ${\bf r}_1$};
\draw[line width=1.0pt,color=black] ((1,0,2) node[left] {\footnotesize ${\bf r}_2$};
\draw[line width=0.7pt,dashed,black] (1,0,3)-- (0,0,3);
\draw[line width=0.7pt,dashed,black] (1,0,0)-- (1,0,1);
\end{tikzpicture}
\end{center}
\end{minipage}
\caption{Grafting $\gamma_1$ along $\gamma_2$}\label{Fig-Grafting}
\end{figure}
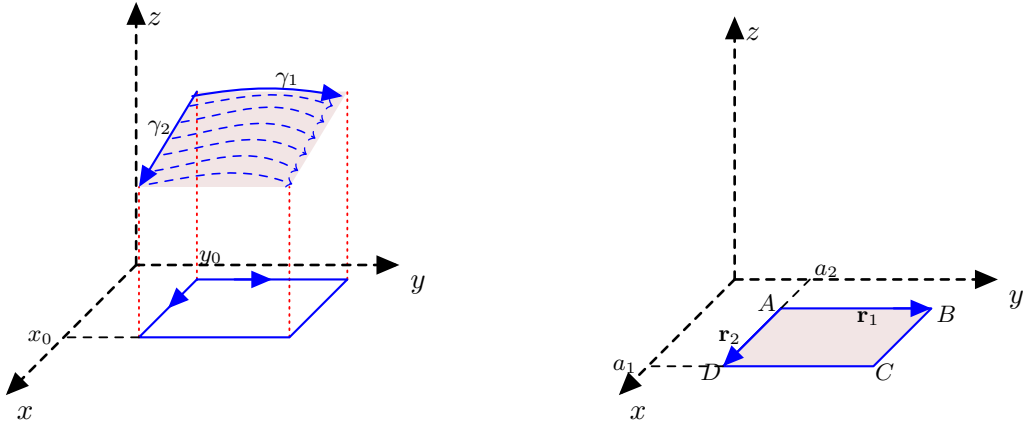
\end{center}

The above example suggests how to extend a function of two variables $F$ over a rectangular planar region, knowing $F$ on one vertical and one horizontal side. Let $A,B,C,D$ be the vertices of the rectangular region and denote $A:=(a_1,a_2).$ Parameterize $AB$ by ${\bf r}_1(s)=(a_1,\alpha (s))$ for $s\in I_1$, and parameterize $AD$ by
${\bf r}_2(t)=(\beta (t),a_2)$ for $t\in I_2$. We then obtain curves in space
\begin{equation}\label{graph}
\gamma_1(s) = (a_1, \alpha (s), F(a_1,\alpha (s))\ \ ,\ \ s\in I_1
\end{equation}
(this is the graph of $F$ over $AB$), and
$$\gamma_2(t) = (\beta (t), a_2, F(\beta (t), a_2))\ \ ,\ \ t\in I_2.$$
We now graft $\gamma_1$ along $\gamma_2$ to get the surface
$$\nu (s,t) = (\beta (t), \alpha (s), F(a_1,\alpha (s)) + F(\beta (t),a_2) -F(a_1,a_2)).$$
For a point $(x,y)=(\beta(t),\alpha (s))$ in the interior of the rectangular region,
we can set
$$\widetilde{F}(x,y) = F(a_1,y) + F(x,a_2) - F(a_1,a_2)$$
Evidently if $F(\downarrow,\uparrow)$ on $AB\cup AD$, meaning $\alpha\uparrow$ and
$\beta\downarrow$, then as well $\widetilde{F}(\downarrow,\uparrow)$.

\begin{corollary}[Filling by Grafting]\label{filling} If $F(\downarrow,\uparrow)$ on $AB\cup AD$, then $\widetilde{F}$ is a well-defined extension on the interior of $ABCD$ with $\widetilde{F}(\downarrow,\uparrow)$.
\end{corollary}

\begin{remark} ({\rm Reverse Grafting})
We can graft a curve $\gamma_1$ at a point $\gamma_2(t)$ by translating $\gamma_1$ so that now the \textit{end point} of $\gamma_1$ ends at $\gamma_2(t)$. We will call this process ``reverse grafting". Recall that normal grafting is about translating $\gamma_1$ so its \textit{starting point} is that of $\gamma_2(t)$. This is used as follows. Suppose now that our rectangle $ABCD$ is parameterized counterclockwise, $\gamma_1$ is the graph of $F$ over $BA$ (with parameterization from $B$ to $A$), and $\gamma_2$ the graph over $AD$. As in Corollary \ref{filling}, a filling of the rectangle can be done by grafting $\gamma_2$ along $\gamma_1$, or by reverse grafting $\gamma_1$ along $\gamma_2$.
\end{remark}

We are now in a position to prove our first extension result.

\begin{theorem}[Convex Case]\label{Th-ConvexCase}
Let $\Omega$ be a compact and convex region in the plane of whose boundary is a piecewise smooth Jordan curve, circumscribed on a rectangle $R$. Then any continuous function $F$ of mixed monotonicity on $\Omega$, with hypothesis as in Lemma \ref{Lem-Sadok1}, admits a nice extension $\widetilde{F}$ on $R$ (see Definition \ref{defext}).
\end{theorem}

\begin{proof}
Based on Lemma \ref{Lem-Origami1}, we start from the point where $\Omega$ is an origami domain and
we parameterize its boundary via ${\bf r}: [0,1]\rightarrow\partial\Omega$. In this case, we need to consider four cases as shown in Fig. \ref{Fig-OrigamiBoundary1} and extend $F$ over the unshaded rectangles,
going counterclockwise around the curve. Each polygonal region labeled (I) through (IV) has a horizontal edge and a vertical edge. Each edge is a portion of the parameterized boundary, positively oriented according to the arrow. We label $\gamma_{top}$ (resp. $\gamma_{bot}$) the curve $F\circ{\bf r}$ restricted to the top edge (resp. bottom edge), and
$\gamma_{ver}$ is $F\circ{\bf r}$ restricted to the vertical edge. The extension $\widetilde{F}(\downarrow,\uparrow)$ is obtained by performing the following filling: \\
$\bullet$ Region (I): grafting $\gamma_{ver}$ along $\gamma_{top}$ (or reverse grafting $\gamma_{top}$ along $\gamma_{ver}$)\\
$\bullet$ Region (II): grafting $\gamma_{top}$ along $\gamma_{ver}$\\
$\bullet$ Region (III): grafting $\gamma_{bot}$ along $\gamma_{ver}$ \\
$\bullet$ Region (IV): grafting $\gamma_{ver}$ along $\gamma_{bot}.$ \\
In both cases (I) and (IV),
the intermediate value theorem shows that grafting doesn't change the range.
This is not the case for the extensions in cases (II) and (III). If we insist on not getting out of the range of $F$,
 we need to use an alternative extension for those cases. Fortunately this is possible in those two cases. Let $\Gamma$ denote any one of these polygonal regions in Fig. \ref{Fig-OrigamiBoundary1} and $\Gamma^c$ the unshaded rectangle. For the rectangle (II), we can set
\begin{equation}\label{min}
\widetilde{F}(x,y) = \begin{cases} F(x,y),& (x,y)\in\Gamma\\
\min\left(F({x^-},y), F(x,y^+)\right),& (x,y)\in \Gamma^c.
\end{cases}\end{equation}
It is easy to check that $\widetilde{F}$ is continuous and $\widetilde{F}(\downarrow,\uparrow)$.
Similarly for rectangle (III), we use the extension
\begin{equation}\label{max}
\widetilde{F}(x,y) = \begin{cases} F(x,y),& (x,y)\in\Gamma\\
\max\left(F({x},y^-), F(x^+,y)\right),& (x,y)\in \Gamma^c.
\end{cases}\end{equation}
\begin{center}
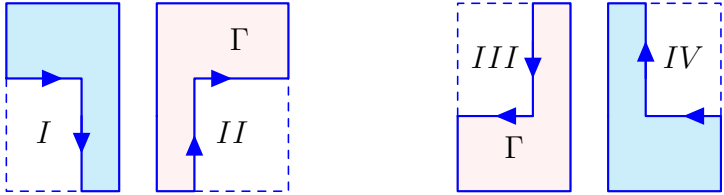
\begin{figure}[!h]
\definecolor{zzttqq}{rgb}{0.6,0.2,0.}
\definecolor{cqcqcq}{rgb}{0.7529411764705882,0.7529411764705882,0.7529411764705882}
\begin{minipage}[t]{0.9\textwidth}
\begin{center}
\begin{tikzpicture}[line cap=round,line join=round,x=1.0cm,y=1.0cm,scale=0.5]
\draw[line width=0.8pt,smooth,fill=cyan,opacity=0.2] (-10.0,1)--(-8,1)--(-8,-2)--(-7,-2)--(-7,3)--(-10,3)--(-10,1);
\draw[line width=0.8pt,smooth,blue] (-10.0,1)--(-8,1)--(-8,-2)--(-7,-2)--(-7,3)--(-10,3)--(-10,1);
\draw[line width=0.6pt,smooth,blue,dashed] (-10.0,1)--(-10,-2)--(-8,-2);
\draw[-triangle 45,line width=0.8pt,smooth,blue] (-9.0,1.0) -- (-8.5,1.0);
\draw[-triangle 45,line width=0.8pt,smooth,blue] (-8.0,0.0) -- (-8.0,-1.);
\draw[shift={(-1.0,0)},line width=0.8pt,xscale=1.0,smooth,fill=pink,opacity=0.2] (-5.0,0)--(-5,-2)--(-4,-2)--(-4,1)--(-1.5,1)--(-1.5,3)--(-5,3)--(-5,0);
\draw[shift={(-1.0,0)},line width=0.8pt,xscale=1.0,smooth,blue] (-5.0,0)--(-5,-2)--(-4,-2)--(-4,1)--(-1.5,1)--(-1.5,3)--(-5,3)--(-5,0);
\draw[line width=0.6pt,smooth,blue,dashed] (-5.0,-2)--(-2.5,-2)--(-2.5,1);
\draw[-triangle 45,line width=0.8pt,smooth,blue] (-5.0,-1.0) -- (-5.0,-0.5);
\draw[-triangle 45,line width=0.8pt,smooth,blue] (-4.5,1.0) -- (-4.0,1.0);
\draw[line width=0.8pt,smooth,fill=pink,opacity=0.2] (2.0,0)--(4,0)--(4,3)--(5,3)--(5,-2)--(2,-2)--(2,0);
\draw[line width=0.8pt,smooth,blue] (2.0,0)--(4,0)--(4,3)--(5,3)--(5,-2)--(2,-2)--(2,0);
\draw[-triangle 45,line width=0.8pt,smooth,blue] (4.0,0.0) -- (3.0,0.);
\draw[-triangle 45,line width=0.8pt,smooth,blue] (4.0,2.0) -- (4.0,1.);
\draw[line width=0.6pt,smooth,blue,dashed] (2.0,0)--(2,3)--(4,3);
\draw[line width=0.8pt,smooth,fill=cyan,opacity=0.2] (6.0,0)--(6,3)--(7,3)--(7,0)--(9,0)--(9,-2)--(6,-2)--(6,0);
\draw[line width=0.8pt,smooth,blue] (6.0,0)--(6,3)--(7,3)--(7,0)--(9,0)--(9,-2)--(6,-2)--(6,0);
\draw[-triangle 45,line width=0.8pt,smooth,blue] (8.5,0.0) -- (8.0,0.);
\draw[-triangle 45,line width=0.8pt,smooth,blue] (7.0,1.0) -- (7.0,2.);
\draw[line width=0.6pt,smooth,blue,dashed] (9.0,0)--(9,3)--(7,3);
\draw[xscale=1.0] (-9,-1) node[above,rotate=0] {\large $I$};
\draw[xscale=1.0] (-4,-1) node[above,rotate=0] {\large $II$};
\draw[xscale=1.0] (-3.8,1.5) node[above,rotate=0] {\large $\Gamma$};
\draw[xscale=1.0] (3,1) node[above,rotate=0] {\large $III$};
\draw[xscale=1.0] (3.5,-1.4) node[above,rotate=0] {\large $\Gamma$};
\draw[xscale=1.0] (8,1) node[above,rotate=0] {\large $IV$};
\end{tikzpicture}
\end{center}
\end{minipage}%
\caption{The various cases for the boundary pieces of an origami domain.}\label{Fig-OrigamiBoundary1}
\end{figure}
\end{center}
These are well-defined extensions, since on the boundary of the curve they coincide with $F$,
and their range stays within the range of $f$.
This concludes the proof and the existence of an extension for convex domains.
\end{proof}

\subsection{Extending $F$ from a semi-convex domain to a rectangular domain} We believe that Theorem \ref{Th-ConvexCase} is true
for all domains $\Omega$, parameterized by ${\bf r}$, so that the derivative of $f=F\circ{\bf r}$ has finitely many zeroes on its domain. This means that the extension to a rectangular region containing $\Omega$ is possible on such domains. We do not verify this claim to this fullest extent, but we restrict ourselves to generalizing the theorem to semi-convex domains. This covers a much wider family of domains $\Omega$ other than the convex ones.
A planar region $\Omega$ with boundary $C=\partial\Omega$ a simple curve, is called semi-convex if a ray emanating from $p\in C$ into $\Omega^c$, which is either horizontal or vertical doesn't intersect the curve again. Part (a) of Fig. \ref{Fig-SemiConvex} illustrates one such domain.

\begin{figure}[!h]
\begin{center}
\begin{tikzpicture}[line cap=round,line join=round,scale=1.2]
\draw [shift={(-5,2)},line width=1.0pt,color=blue]  plot[domain=-1.4:1.45,variable=\t]({\t},{-0.2*\t*\t+1.8});
\draw [shift={(-5,2)},line width=1.0pt,color=blue]  plot[domain=0.0:1.37,variable=\t]({0.5*\t*\t+0.5},{\t});
\draw [shift={(-5,2)},line width=1.0pt,color=blue]  plot[domain=-0.8:0.0,variable=\t]({1.4*\t*\t+0.5},{\t});
\draw [shift={(-5,2)},line width=1.0pt,color=blue]  plot[domain=-1.4:1.4,variable=\t]({-0.2*\t*\t-1},{\t});
\draw [shift={(-5,2)},line width=1.0pt,color=blue]  plot[domain=-1.4:1.4,variable=\t]({\t},{-0.2*\t*\t-1});
\draw[shift={(-5,2)},line width=1.0pt,blue] (1.4,-1.4)--(1.4,-0.8);
\draw [name path=G1,shift={(3,2)},line width=0.8pt,color=blue]  plot[domain=-1.6:1.6,variable=\t]({0.2*\t*\t+1},{\t});
\draw[shift={(-5,0.5)}] (0.0,0) node[below,rotate=0] {(a)};
\draw[shift={(4.5,0.5)}] (0,0) node[below,rotate=0] {(d)};
\draw[shift={(-5,2.2)}] (0,0) node[below,rotate=0] {\Large $\Omega$};
\draw[-triangle 45,shift={(3,2)},line width=0.8pt,smooth,blue] plot[domain=0.7:1.1,variable=\t]({0.2*\t*\t+1},{\t});
\draw[-triangle 45,shift={(3,2)},line width=0.8pt,smooth,blue] plot[domain=-1.1:-0.7,variable=\t]({0.2*\t*\t+1},{\t});
\fill[line width=0.5pt,shift={(-5,2)},Maroon!20,opacity=.5] (1.38,-1.38) -- (1.38,-0.8) -- (1.13,-0.7) -- (0.7,-0.4) -- (0.5,-0.1) -- (0.5,-0.0) -- (0.52,0.2) -- (0.62,0.5) -- (0.82,0.8) -- (1.16,1.15) -- (1.22,1.2) -- (1.28,1.25) -- (1.35,1.3) -- (1.45,1.4) -- (1.3,1.46) -- (1,1.6)--(0.80,1.67) -- (0.61,1.72) -- (0.4,1.77) -- (0.2,1.79) -- (0.,1.8) -- (-0.2,1.79) -- (-0.40,1.76) -- (-0.60,1.73) -- (-0.8,1.7) -- (-1.,1.6) -- (-1.2,1.5) -- (-1.38,1.4) -- (-1.29,1.20) -- (-1.2,1.) -- (-1.13,0.81) -- (-1.074,0.61) -- (-1.03,0.4) -- (-1.01,0.2) -- (-1.,0.) -- (-1.01,-0.2) -- (-1.04,-0.42) -- (-1.09,-0.67) -- (-1.15,-0.88) -- (-1.24,-1.103) -- (-1.38,-1.38) -- (-1.19,-1.29) -- (-1.,-1.2) -- (-0.8,-1.13) -- (-0.6,-1.07) -- (-0.4,-1.03) -- (-0.2,-1.01) -- (0.,-1.) -- (0.2,-1.01) -- (0.4,-1.03) -- (0.59,-1.07) -- (0.8,-1.13) -- (1.,-1.2) -- (1.20,-1.29) -- cycle;
\draw[shift={(-1.2,2.8)}] (0,0) node[below,rotate=0] { $f(\downarrow)$};
\draw[shift={(-1.5,1.7)}] (0,0) node[below,rotate=0] { $f(\uparrow)$};
\draw [shift={(-3,2)},line width=0.8pt,color=blue]  plot[domain=-1.4:1.4,variable=\t]({0.2*\t*\t+1},{\t});%
\draw[-triangle 45,shift={(-3,2)},line width=0.8pt,smooth,blue] plot[domain=0.7:1.1,variable=\t]({0.2*\t*\t+1},{\t});%
\draw[-triangle 45,shift={(-3,2)},line width=0.8pt,smooth,blue] plot[domain=-1.1:-0.7,variable=\t]({0.2*\t*\t+1},{\t});%
\draw[shift={(-2,2)},line width=0.6pt,red,dashed] (1,0)--(0,0);
\draw[-triangle 45,shift={(-2,2)},line width=0.6pt,red] (0.4,0.2)--(0.4,1);
\draw[-triangle 45,shift={(-1.5,1)},line width=0.6pt,red] (0.7,0)--(0,0);
\draw[shift={(-1.4,0.5)}] (0,0) node[below,rotate=0] {(b)};
\draw[shift={(1.8,2.8)}] (0,0) node[below,rotate=0] { $f(\downarrow)$};
\draw[shift={(1.8,1.7)}] (0,0) node[below,rotate=0] { $f(\uparrow)$};
\draw [shift={(0,2)},line width=0.8pt,color=blue]  plot[domain=-1.4:1.4,variable=\t]({0.2*\t*\t+1},{\t});%
\draw[-triangle 45,shift={(0,2)},line width=0.8pt,smooth,blue] plot[domain=0.7:1.1,variable=\t]({0.2*\t*\t+1},{\t});%
\draw[-triangle 45,shift={(0,2)},line width=0.8pt,smooth,blue] plot[domain=-1.1:-0.7,variable=\t]({0.2*\t*\t+1},{\t});%
\draw[shift={(1,2)},line width=0.6pt,red,dashed] (1,0)--(0,0);
\draw[-triangle 45,shift={(1,2)},line width=0.6pt,red] (0.4,0.2)--(0.4,1);
\draw[-triangle 45,shift={(1.4,1)},line width=0.6pt,red] (0,0.5)--(0,-0.3);
\draw[shift={(1.5,0.5)}] (0,0) node[below,rotate=0] {(c)};
\draw[shift={(4,2)},line width=0.6pt,red] (0.5,-1.6)--(0.5,0)--(0,0)--(0.5,0.5)--(0.5,1.6);
\fill[Maroon!20,opacity=.5,shift={(4,2)},] (0.5,-1.6) -- (0.5,0)--(0,0)--(0.01,-0.2)--(0.1,-0.8)--(0.2,-1)--(0.5,-1.6);
\fill[Maroon!20,opacity=.5,shift={(4,2)},] (0.5,1.6) -- (0.5,0.5)--(0,0)--(0.01,0.2)--(0.1,0.8)--(0.2,1)--(0.5,1.6);
\end{tikzpicture}
\end{center}
\caption{A semi-convex domain is shown in Part (a). Next is shown a piece $C_1\cup C_2$ of the boundary of $\Omega$ oriented positively, with $f(t)=F(r(t))$ changing monotonicity from $C_1$ to $C_2$. Parts (b) and (c) show
two different ray extensions, while (d) shows the adjustment to make to these extensions, before filling the sector. As explained in the text, (b) is not a valid extension if monotonicity is to be preserved in the extension.}\label{Fig-SemiConvex}
\end{figure}
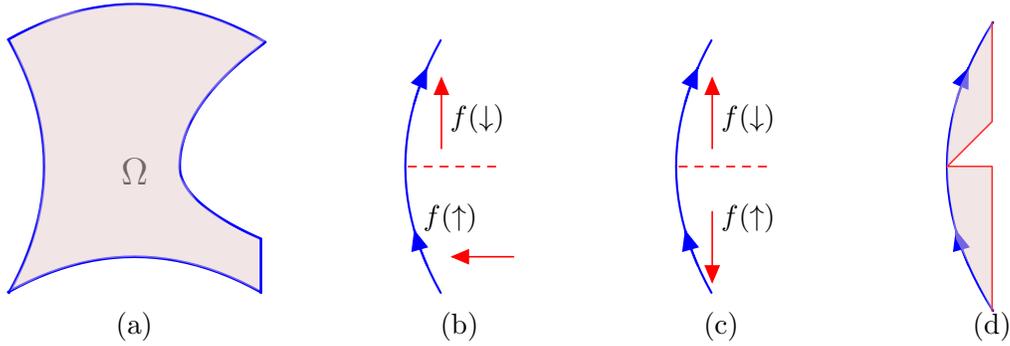

\begin{theorem}[Semi-Convex Case]
Suppose $\Omega$ is semi-convex and $F(\downarrow, \uparrow)$ with hypothesis as in Lemma \ref{Lem-Sadok1}. Then $F$ has a nice extension to $\widetilde{F}(\downarrow, \uparrow)$ on a larger rectangular domain containing $\Omega$.
\end{theorem}

\begin{proof}
Recall that we are writing our proofs based on mixed monotonicity of the form $F(\downarrow, \uparrow)$.
By hypothesis $f:= F\circ{\bf r}$, where ${\bf r}$ is a parametrization of $C$, is monotonic on a finite number of arcs $C_1,\ldots, C_n$, $C=C_1\cup\cdots\cup C_n$ (see Lemma \ref{Lem-Sadok1}).
We assume the arcs are parameterized positively (that is the region $\Omega$ is always to the lefthand side
of the curve as we travel along increasing $t$), and that when they intersect, they do so at an endpoint.
We can proceed as in the convex case, extending $F$ first on an origami domain.
To get to the origami domain, we extend $F$ on a box containing $C_i$ for all $i$ by performing either a horizontal or a vertical ray extension. Here however we might run into the following problems illustrated by Parts (b)
and (c) of Fig.  \ref{Fig-SemiConvex}.

Indeed, suppose $C_1,C_2$ are consecutive segments of $C$, $f_1=F\circ{\bf r}_1$ is $\uparrow$ while
$f_2=F\circ{\bf r}_2$ is $\downarrow$, and rule out at this stage linear segments perpendicular to each other
(we refer to Fig.  \ref{Fig-SemiConvex}).
We can then extend to the right
of the indicated boxes as in Lemma \ref{Lem-Sadok2} by performing the ray extension shown by the projection arrows. When this is done, we might obtain a discontinuity along the ray emanating from $p=C_1\cap C_2$.
Along that ray, $\widetilde{F}_1$ and $\widetilde{F}_2$ the corresponding extensions for $C_1$ and $C_2$ do not necessarily agree. To go around this problem, we extend as before but leave a sector as shown in Part (d) of Fig. \ref{Fig-SemiConvex}. This extension is well-defined and continuous, and so now we have to extend over the sector.
A sector has two parameterized edges: the first edge $E_1$ and the second edge $E_2$, where first and second are with respect to increasing $t$ in the parametrization. The second issue to pay attention to is that the value
of $\widetilde{F_1}$ at a point of $E_1$ might not be less than the value of $\widetilde{F_2}$ at the point immediately above it in $E_2$.
No sector filling will ensure that $\tilde F(\downarrow,\uparrow)$, and so one has to choose the ray extensions carefully from the start to avoid this issue. This can be done simply by choosing the right combination of ray
extensions. For example, in
Fig. \ref{Fig-SemiConvex}, extension (c) is a valid extension (in the sense that $\widetilde{F_1}(P)\leq
\widetilde{F_2}(Q)$ if $Q$ is immediately above $P$), but (b) is not necessarily.

We can therefore assume we have extended $F$ as far as Part (d) of Fig. \ref{Fig-SemiConvex},
with $F(\downarrow, \uparrow)$, and now
we wish to extend over the sector. There are cases to consider, illustrated in
Fig. \ref{Fig-Sectors}, which cover the various allowable combinations of $f_1\uparrow\downarrow$ on $E_1$ and $f_2\uparrow\downarrow$ on $E_2$ that occur.
We treat in details the case of Fig. \ref{Fig-Sectors},(d), the other three being similar.

Suppose we have a sector as in Part (d) of Fig. \ref{Fig-Sectors}.  Note that the function $f_1=F\circ {\bf r}_1$
must be increasing $\uparrow$ on $E_1$. If $f_2$ is decreasing, then we can proceed as follows.
Let $(x,y)$ be a point in this sector, $(x,y^-)$ its vertical projection in $E_1$ and
$(x,y^+)$ its vertical projection in $E_2$. We can parameterize the segment
$[(x,y^-), (x,y^+)]$ by arc-length so that a point there is of the form
$$(x,y) = \left(1-{y-y^-\over y^+-y^-}\right)(x,y^-) + {y-y^-\over y^+-y^-}(x,y^+)$$
and consequently we can define a linear extension of $F$ as follows
$$\widetilde F(x,y) = \left(1-{y-y^-\over y^+-y^-}\right)F(x,y^-) + {y-y^-\over y^+-y^-}F(x,y^+).$$
Such an extension is continuous by construction, and we can verify that $\widetilde F(\downarrow,\uparrow);$ further, the range of $\widetilde F$ is within the range of $F$.

Suppose now that $f_2\uparrow$ (still dealing with Figure \ref{Fig-Sectors} (d)). The linear extension above doesn't yield necessarily the same mixed monotonicity extension.
Instead and in this case, we can graft $\gamma_2$ along $\gamma_1$ where as is now understood $\gamma_i$ is the restriction of $F$ over parameterized $E_i$ (see Eq. \eqref{graph} and Fig. \ref{Fig-Grafting}). In this case as well, the range of $\widetilde F$ stays within the range of $F$.

The other extensions in cases (a),(b) and (c) work the same. Note that for (a), $f_1$ must be decreasing and $f_2$ is increasing, because of the nature of $F(\downarrow,\uparrow)$. Here linear extension does the job. Same for (c).

Finally, by carrying out these extensions over sectors we end up extending $F$ over a domain which is again an Origami domain, meaning all edges are either vertical or horizontal. By trying to extend this new $F$ to a rectangular domain, as we've done for the convex case, we run into the novel situation where we have to fill a region surrounded by three sides as in Part (a) of Fig. \ref{Fig-Redundancy}. As indicated in that figure, we
extend first to $\widetilde{F}_1$ on the rectangle minus a sector,
then fill-in for the remaining sector. Remember that in the first step,
we need to ensure that $\widetilde{F}_1(Q)\geq F(P)$ for
$P\in E_1$ and $Q$ immediately above it in segment $E_2$. This is possible using either Eq. \eqref{min} or Eq. \eqref{max} depending on the situation. For example in Fig. \ref{Fig-Redundancy}, the restriction $f$ of
$F$ to the left edge is necessarily increasing and its restriction to the top edge is necessarily decreasing.
In this case, we can choose
$\widetilde{F}_1(x,y) = \min (F(x^-,y), F(x,y^+))$. Notice that if $Q=(x,y)$ as shown in that figure,
then $P(x,y^-)$ and
\begin{eqnarray*}
\widetilde{F}_1(Q) &=& \min (F(x^-,y), F(x,y^+))\\
&\geq& \min (F(x^-,y), F(x,y^-)) \ \ \hbox{since}\
F(x,y^-)\leq F(x,y^+)\ \hbox{by hypothesis}\\
&\geq & F(x,y^-)=F(P)\ \ \hbox{since}\ f\uparrow\ \hbox{going from $P$ to $(x^-,y)$}
\end{eqnarray*}
Once $\widetilde{F}_1$ is constructed, one extends over the remaining sector as done previously.
There is one point however that needs extra care: the extension $\widetilde{F}_1$, when restricted
to the parameterized edge containing $Q$ in Fig. \ref{Fig-Redundancy}, is not necessarily monotonous.
When it is monotonous, then we proceed indeed as previous (linear extension
or grafting). Generally however, that parameterized edge
could break into a finite number of segments on which the restriction of $\widetilde{F}_1$ is increasing or decreasing. In that case, one has to combine some suitable choices of ray extensions, linear extensions
and grafting to get the desired extension over the sector. We skip writing the details as they are now familiar but tedious.
\end{proof}

\begin{figure}[!h]
\begin{center}
\begin{tikzpicture}[line cap=round,line join=round,scale=0.9]
\draw[line width=0.8pt,blue] (-7,2)--(-5,2)--(-5,4)--(-6,4)--(-6,2.5)--(-7.0,4)--(-7,2);
\draw[line width=0.8pt,fill=pink,opacity=0.2] (-7,2)--(-5,2)--(-5,4)--(-6,4)--(-6,2.5)--(-7.0,4)--(-7,2);
\draw[line width=0.8pt,blue,shift={(0,0)}] (-4,2)--(-2,2)--(-2,4)--(-4,4)--(-4,3)--(-3.0,3)--(-4,2);
\draw[line width=0.8pt,fill=pink,opacity=0.2,shift={(0,0)}] (-4,2)--(-2,2)--(-2,4)--(-4,4)--(-4,3)--(-3.0,3)--(-4,2);
\draw[line width=0.8pt,blue,shift={(0,0)}] (2,2)--(4,2)--(4,3.0)--(3,3)--(4,4)--(2,4)--(2,2);
\draw[line width=0.8pt,fill=pink,opacity=0.2,shift={(0,0)}] (2,2)--(4,2)--(4,3.0)--(3,3)--(4,4)--(2,4)--(2,2);
\draw[line width=0.8pt,blue,shift={(0,0)}] (-1,2)--(0,2)--(0,3)--(1,2)--(1,4)--(-1,4)--(-1,2);
\draw[line width=0.8pt,fill=pink,opacity=0.2,shift={(0,0)}] (-1,2)--(0,2)--(0,3)--(1,2)--(1,4)--(-1,4)--(-1,2);
\draw[shift={(3,0)}] (-9.0,2) node[below,rotate=0] {(a)};
\draw[shift={(3,0)}] (-6.0,2) node[below,rotate=0] {(b)};
\draw[shift={(9,0)}] (-6.0,2) node[below,rotate=0] {(d)};
\draw[shift={(6,0)}] (-6.0,2) node[below,rotate=0] {(c)};
\draw[-triangle 45,shift={(3,0)},line width=0.8pt,smooth,blue] (-8,4.0) -- (-9,4.0);
\draw[-triangle 45,shift={(3,0)},line width=0.8pt,smooth,blue] (-5,4.0) -- (-6,4.0);
\draw[-triangle 45,shift={(9,0)},line width=0.8pt,smooth,blue] (-5,4.0) -- (-6,4.0);
\draw[-triangle 45,shift={(6,0)},line width=0.8pt,smooth,blue] (-5,4.0) -- (-6,4.0);
\end{tikzpicture}
\end{center}
\caption{Extending from semi-convex domain to the origami domain reduces to extending over sectors, with cases labeled (a) through (d)} \label{Fig-Sectors}
\end{figure}
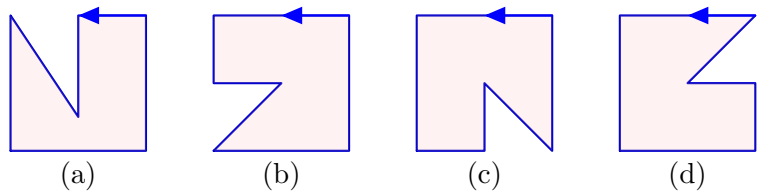

 \begin{figure}[htb]
 \begin{center}
\begin{tikzpicture}[line cap=round,line join=round,scale=0.6]
\draw[line width=0.8pt,blue] (-7,2)--(-3,2)--(-3,3)--(-5,3)--(-5,6)--(-3.0,6)--(-3,7)--(-7,7)--(-7,2);
\draw[line width=0.8pt,fill=pink,opacity=0.2] (-7,2)--(-3,2)--(-3,3)--(-5,3)--(-5,6)--(-3.0,6)--(-3,7)--(-7,7)--(-7,2);
\draw[-triangle 45,shift={(0,0)},line width=0.8pt,smooth,blue] (-3,3.0) -- (-4,3.0);
\draw[-triangle 45,shift={(0,0)},line width=0.8pt,smooth,blue] (-5,4.0) -- (-5,5.0);
\draw[-triangle 45,shift={(0,0)},line width=0.8pt,smooth,blue] (-5,6.0) -- (-4,6.0);
\draw[line width=0.8pt,blue,shift={(8,0)}] (-7,2)--(-3,2)--(-3,3)--(-5,3)--(-3.0,6)--(-3,7)--(-7,7)--(-7,2);
\draw[line width=0.8pt,fill=pink,opacity=0.2,shift={(8,0)}] (-7,2)--(-3,2)--(-3,3)--(-5,3)--(-5,6)--(-3.0,6)--(-3,7)--(-7,7)--(-7,2);
\draw[line width=0.8pt,fill=pink,opacity=0.2,shift={(8,0)}] (-5,3)--(-5,6)--(-3,6)--(-5,3);
\draw[-triangle 45,shift={(8,0)},line width=0.8pt,smooth,blue] (-3,3.0) -- (-4,3.0);
\draw[-triangle 45,shift={(8,0)},line width=0.8pt,smooth,blue] (-5,3.0) -- (-4,4.5);
\draw[shift={(4,0)}] (-9.0,2) node[below,rotate=0] {(a)};
\draw[shift={(3,3)}] (-9.0,2) node[below,rotate=0] {$F$};
\draw[shift={(9,0)}] (-6.0,2) node[below,rotate=0] {(b)};
\draw[shift={(9.5,4)}] (-6.0,2) node[below,rotate=0] {$\widetilde{F}_1$};
\draw[shift={(4,1)}] (0.0,2.1) node[below,rotate=0] {\tiny $P$};
\draw[shift={(4,3)}] (0.15,1.7) node[below,rotate=0] {\tiny $Q$};
\draw[shift={(4,1)},line width=0.7pt,dashed,black] (0,2)-- (0,3);
\end{tikzpicture}
\end{center}
\caption{Filling a rectangle where $F$ is known on three sides. Extending partially by $\widetilde{F}_1$
as shown in (b) then extending over the resulting sector.}\label{Fig-Redundancy}
\end{figure}
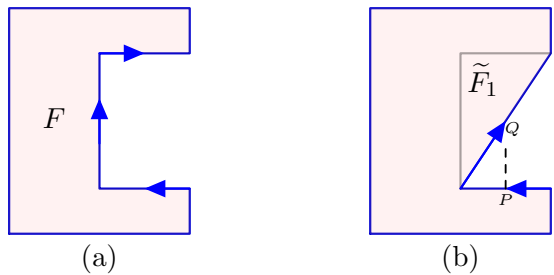


\section{Posets and the embedding technique}\label{embeddingsection}


The embedding technique can be used to study the asymptotic behaviour of the orbits in Eq. (\ref{Eq-DE1}) whenever $F $ is of mixed monotonicity and assumes an invariant box. In its simplest form, this approach  can be summarized as follows \cite{Go-Ha1994,Sm2006}.

\begin{proposition}\label{Pr-embed}
Suppose  $F: \Omega:=[a,b]^2\rightarrow [a,b]$ and $F(\uparrow,\downarrow).$ Define
the symmetric map $G(x,y)= (F(x,y),F(y,x)).$ Then each orbit of $ X_{n+1}=G(X_n)$ must converge to a fixed point of $G.$ In particular, if $G$  has a unique fixed point $(x^*,x^*),$ then it must be a global attractor of $ X_{n+1}=G(X_n).$
\end{proposition}

Notice that the diagonal of $\Omega,$ i.e., $\{(x,x):\, a\leq x\leq b\}$ is invariant under $G.$ However, convergence established under $G$ is not generally helpful in establishing convergence under $F$ since
forward obits of $F$, $\mathcal{O}_F^+(x_0,x_{-1})=\{x_{-1},x_0,x_1,\ldots\}$, where $x_{n+1}=F(x_n,x_{n-1})$ can be scrambled under $G$.
 The following example illustrates such an unfavorable situation.

\begin{example}\label{xf(y)}\rm
Let $F(x,y)=xf(y),$ where $f$ is bounded and decreasing with $f(0)>1,$ then the Jacobian matrix of $G(x,y)=(F(x,y),F(y,x))$ is given by
$$J\left(\left[
           \begin{array}{c}
             F(x,y) \\
             F(y,x) \\
           \end{array}
         \right]\right)=\left[
                   \begin{array}{cc}
                     f(y) & xf^\prime(y) \\
                     yf^\prime(x) & f(x) \\
                   \end{array}
                 \right].
                 $$
Thus, at the positive equilibrium $x^*$ of $F,$ we obtain
$$J(x^*,x^*)=\left[
                   \begin{array}{cc}
                     1 & x^*f^\prime(x^*) \\
                     y^*f^\prime(x^*) & 1 \\
                   \end{array}
                 \right].
                 $$
The eigenvalues of the Jacobian matrix at $(x^*,x^*)$ are
$$\lambda_1=1-x^*f^\prime(x^*)>1\quad \text{and}\quad \lambda_2=1+x^*f^\prime(x^*)<1.$$
Hence, $(x^*,x^*)$ cannot be stable, and therefore, it cannot globally attract $\Omega$ under the iteration of $G.$ On the other hand, when $-2\leq x^*f^\prime(x^*)<0,$  $(x^*,x^*)$ is a saddle, and consequently, a one dimensional manifold of $\Omega$ is attracted to $(x^*,x^*).$ Indeed, it is easy to show in this case that the diagonal of $\Omega$ is attracted under $G$ to $(x^*,x^*).$ In conclusion, if $F(x,y)=xf(y),$ then $(x^*,x^*)$ cannot attract all of $\Omega$ in Corollary \ref{Prop-1}, and therefore, there must be at least two other asymmetric fixed points of $G.$
\end{example}

The rest of the section elaborates on alternative embeddings (i.e. other choices of maps $G$ that can play a similar role as in Proposition \ref{Pr-embed}). The fixed points of the embedded map $G$ play a key role in Proposition \ref{Pr-embed}, and they are the solutions of the simultaneous equations $F(x,y)=x$ and $F(y,x)=y.$ Obviously,  the fixed points of $F$ are solutions. We call a solution of $(F(x,y),F(y,x))=(x,y)$ that does not satisfy $x=y$  \emph{an artificial fixed point} of $F.$
\\

A Riesz space is a vector space endowed with a partial order $\leq$ which makes it a lattice and satisfies translation invariance and positive homogeneity \cite{Lu-Za1971}. For example, $\mathbb R^n$ with $\leq _{se}$, $\leq _{ne}$, $\leq _{sw}$ or $\leq _{nw}$ ordering is a Riesz space.
The north-east partial ordering corresponds to the componentwise order (i.e. $(x_1,x_2)\leq (y_1,y_2)$ if and only if $x_1\leq y_1$ and $x_2\leq y_2$), which is the most recognizable Riesz partial ordering on $\mathbb R^n.$
A Riesz space is Archimedean if it also satisfies the following condition: for every pair of positive elements $x$ and $y$, there exists an integer n such that $nx\geq y$. Order convergence and topological convergence coincide on Archimedean Riesz spaces, and in fact much more is true for finite dimensional Riesz space.
The following can be found in \cite{Lu-Za1971}.

\begin{theorem}
Let $(E,\leq)$ be a Riesz space of finite dimension $n\geq 1$. Then $E$ is Archimedean if and only if it is isomorphic to $\mathbb R^n$ under its componentwise order.
\end{theorem}

This means in particular that the analog of the ``monotone sequence theorem" applies
in the Archimedean case.

\begin{corollary}\label{riesz} In a finite dimensional Archimedean Riesz space, a bounded increasing chain $x_1\leq x_2\leq \cdots$ converges to its least upper bound. The convergence is topological.
\end{corollary}

A function $f:X\rightarrow X$ is order preserving with respect to the partial ordering of $X$ if whenever $x\leq y$,  $f(x)\leq f(y)$.
The following is an immediate consequence of Corollary \ref{riesz}.

\begin{proposition}\label{partialorder}
Let $X$ be a closed subspace of an Archimedean Riesz space $(E,\leq)$ with a minimal element $0$ and maximal element $1$ for $(X,\leq)$. Let $f: X\rightarrow X$ be any continuous order preserving function. Then there are fixed points $a^*$ and $b^*$ of $f$, so that for any $x\in X$, we have
$$a^*\leq \liminf f^n(x)\ ,\ \limsup f^{n}(x)\leq b^*.$$
\end{proposition}

\begin{proof}
Since $0$ is minimal, $0\leq f(0)$, and so using the order preserving
properties of $f$, $\{f^{n}(0)\}$ is an increasing chain in $X$.
It must converge to its least upper bound, say $a^*$ in $X$.
Since $f$ is continuous and the convergence is topological,
this implies that $a^*$ must be a fixed point of $f$.
Similarly, $\{f^n(1)\}$ must be decreasing converging to a fixed point, say $b^*$ in $X$.
Finally and for any $x\in X$, $0\leq x\leq 1$ so that $f^{n}(0)\leq f^{n}(x)\leq f^{n}(1)$ for all $n\geq 1$. From this the claim follows.
 Notice that when $a^*=b^*$, $\{f^n(x)\}$ converges to a fixed point of $f$, independently of $x\in X$.
\end{proof}

\begin{corollary}\label{Prop-1} \cite{Go-Ha1994,Sm2006}
Let $\Omega:=[a,b]^2,$ and suppose $F:\,\Omega \rightarrow [a,b]$ is a
continuous function that satisfies $F(\uparrow,\downarrow).$  Define
$G:\, \Omega\rightarrow \, \Omega$ to be the symmetric map
$G(x,y))=(F(x,y),F(y,x)).$
For any initial condition $X=(x,y)\in \Omega,$ the sequence
$\{G^{n}(X)\}$ is either convergent to a fixed point $X^{\ast }=(x^*,y^*)$ of $G$ or
$$(x^*,y^*)\leq_{se}\, \liminf G^n(x,y),\, \limsup G^n(x,y)\leq_{se}\, (y^*,x^*).$$
\end{corollary}

\begin{proof}
$G$ is order preserving for the $\leq _{se}$ order.
This is easy to observe since for $(x,y)\leq_{se}(u,v),$ the two facts
\[
F(x,y)\leq F(u,y)\leq F(u,v)\qquad \text{and}\qquad F(y,x)\geq F(v,x)\geq F(v,u)\]
imply
$G(x,y)=(F(x,y),F(y,x))\leq _{se}(F(u,v),F(v,u))=G(u,v).$
Now, set $X=[a,b]^2$, $f = G$, the minimal element $0=(a,b)$, the maximal element $1=(b,a)$. We have $a^*= lub\{G^{n}(a,b)\}$ and $b^*=glb\{G^n(b,a)\}$. Clearly
if $a^* =(x^*,y^*)$, then $b^* = (y^*,x^*)$.
\end{proof}

Example \ref{xf(y)} above prompts us to seek alternative choices of a symmetric map $G$ that makes a visible connection between the orbits of $F$ and the orbits of the symmetric system. We provide two options below (Corollaries \ref{Th-Ahmad2}
and \ref{Th-Ahmad}). A third option for defining $G$ has been considered in \cite{Sm2006}.

\begin{corollary} \label{Th-Ahmad2} Suppose that $F:\,\Omega \rightarrow [a,b]$ is
continuous and satisfies $F(\uparrow,\downarrow).$ Consider the map
$G:\,\Omega\times \Omega\rightarrow \,\Omega\times \Omega$ defined by
\begin{equation}\label{Eq-SymmetricMapSec4}
 G((x,y),(u,v))=((F(x,y),u),(F(u,v),x)).
\end{equation}
Then, for each initial conditions $X,Y\in \Omega,$ the omega limit set $\omega((X,Y),G)$ is either a fixed point $(X^{\ast },Y^{\ast})$ of $G$ or belongs to an interval of the form
$I:=\left[ \left(X^{\ast },Y^{\ast }\right),\left(Y^{\ast },X^{\ast }\right)\right] _{se}$ in which the endpoints are fixed points of $G.$
\end{corollary}

\begin{proof}
Consider $\Omega=[a,b]\times[a,b]$  and $\Omega\times \Omega$ with the southeast
order $\leq _{se}$. Define the map
\begin{equation*}\label{Eq-gBasedOnT2}
g:\Omega\times \Omega\rightarrow \Omega\quad \text{as}\quad g((x,y),(u,v))=(F(x,y),u).
\end{equation*}
It can be checked that $g(\uparrow,\downarrow)$, and by writing
$$ G((x,y),(u,v))=(g((x,y),(u,v)),g((u,v),(x,y))) $$
we can see that $G$ is order preserving with respect to
$\leq_{se}$.
Set $A:=(a,b)$ and $B:=(b,a).$
If $X=\Omega\times\Omega$, then $(X,\leq_{se})$ is a closed subspace of
an Archimedean Riesz space with minimal element $(A,B)$ and
maximal element $(B,A)$. The conclusion follows from Proposition \ref{partialorder}.
\end{proof}


The map $G$ as defined in Eq. \eqref{Eq-SymmetricMapSec4} is of particular
interest to us since the connection between the orbits of $G$ and the orbits of $F$ becomes visible. The next proposition shows how to utilize the orbits of $G$ in squeezing the orbits of $F,$ which has the advantage of establishing global stability for $F$ when global stability for $G$ is achieved.

\begin{proposition}\label{Pr- LipschitzG2}
Let $X=(x,y)$ and $Y=(u,v)\in \Omega.$ Each of the following holds true:
\begin{description}
\item{(i)} If $x\leq v\leq y\leq u, u\leq F(u,v)$ and $ F(x,y)\leq x,$ then
$$G(X,Y)\leq_{se}(X,Y)\leq_{se}(X,X)\leq_{se}(Y,X)\leq_{se}G(Y,X).$$
In particular, the same inequalities hold true if $(x,y)=(v,u)$ and $F(x,y)\leq x\leq y\leq F(y,x).$

\item{(ii)} If  $v\leq x\leq u\leq y, F(u,v)\leq u$ and $x\leq F(x,y),$ then
$$(X,Y)\leq_{se}G(X,Y)\leq_{se}G(X,X)\leq_{se}G(Y,X)\leq_{se}(Y,X).$$
In particular, the same inequalities hold true if $(x,y)=(v,u)$ and $x\leq F(x,y),F(y,x)\leq y.$
\end{description}
\end{proposition}
\begin{proof}
For $X=(x,y)$ and $Y=(u,v),$ $G(X,Y)\leq_{se}(X,Y)$  is equivalent to $(F(x,y),u)\leq_{se} X$ and $Y\leq_{se} (F(u,v),x).$ Therefore,
\begin{equation*}\label{In-InequalitiesSec4a}
G(X,Y)\leq_{se}(X,Y)\qquad \Leftrightarrow\qquad
\begin{cases}
F(x,y)\leq x,\\
u\leq F(u,v),\\
y\leq u,\;\text{and}\; x\leq v
\end{cases}
\end{equation*}
and
\begin{equation*}\label{In-InequalitiesSec4b}
(X,Y)\leq_{se}G(X,Y)\qquad \Leftrightarrow\qquad
\begin{cases}
x\leq F(x,y),\\
F(u,v)\leq u,\\
v\leq x\;\text{and}\; u\leq y.\\
\end{cases}
\end{equation*}
All claims follow immediately.
\end{proof}

Our next corollary offers an alternative to corollary \ref{Th-Ahmad2}, but in higher dimension.

\begin{corollary} \label{Th-Ahmad} Let $F:\;\Omega \rightarrow [a,b]$ be a
continuous function such that $F(\uparrow,\downarrow)$,
and consider the symmetric map
$G:\;\Omega ^{2}\times \Omega ^{2}\rightarrow \Omega ^{2}\times \Omega^{2}$ defined by
\begin{equation*}\label{Eq-gBasedOnAhmad3}
G((X,Y),(U,V))=(((F(X),F(V)),X),((F(U),F(Y)),U)).
\end{equation*}
Then for each initial condition $(X,Y),(U,V)\in \Omega ^{2},$ the sequence
$\{G^{n}((X,Y),(U,V))\}$ either converges to a fixed point $((X^{\ast },X^{\ast
}),(U^{\ast },U^{\ast }))$ of $G$, or lies in an interval between two such fixed points.
\end{corollary}

\begin{proof}
The proof follows the same lines once it is shown that
$G(\uparrow)$ with respect to $\leq_{se}$ on $\Omega^4 = \Omega^2\times\Omega^2.$
This can be done as follows: give $\Omega$  and $\Omega^4$  the southeast
ordering $\leq _{se},$ and give $\Omega^2$ the northeast ordering $\leq _{ne}.$ Define the maps
\begin{equation*}\label{Eq-gBasedOnAhmad1}
g_1:\; \Omega \times \Omega \rightarrow \Omega   \quad \text{as}\quad g_1((x,y),(u,v))=(F(x,y),F(u,v)).
\end{equation*}
and
\begin{equation*}\label{Eq-gBasedOnAhmad2}
g:\Omega ^{2}\times \Omega ^{2}\rightarrow \Omega ^{2} \quad \text{as}\quad g((X,Y),(U,V))=(g_1(X,V),X).
\end{equation*}
We claim that $g_1(\uparrow,\downarrow)$ and $g(\uparrow,\downarrow).$
Consider $X^{\prime }=(x^{\prime
},y^{\prime })$ and let $(X,Y),(X^{\prime },Y)\in \Omega \times \Omega $
such that $X\leq _{se}X^{\prime }.$ Then
\[
g_1(X,Y)=(F(X),F(Y))\;\text{ and }\;g_1(X^{\prime },Y)=(F(X^{\prime }),F(Y)).
\]%
Since $x\leq x^{\prime }$ and $y^{\prime }\leq y$, we obtain
$$
F(X)=F(x,y)\leq F(x^{\prime },y)\leq F(x^{\prime },y^{\prime })=F(X^{\prime
}),
$$
which implies $g_1(X,Y)\leq _{se}g_1(X^{\prime },Y).$  Next, consider $Y^{\prime
}=(u^{\prime },v^{\prime }),$ and let $(X,Y),(X,Y^{\prime })\in \Omega
\times \Omega $ such that $Y\leq _{se}Y^{\prime }$, then $u\leq u^{\prime
},v\geq v^{\prime }$. Thus,
\[
F(Y)=F(u,v)\leq F(u^{\prime },v)\leq F(u^{\prime },v^{\prime })=F(Y^{\prime
}),
\]%
which implies $g_1(X,Y^{\prime })\leq _{se}g_1(X,Y).$
Similarly one shows that $g(\uparrow,\downarrow)$, from which we deduce that $G(\uparrow)$. Finally it is straightforward to find that the fixed
points of $G$ are obtained from the solutions of $(F(x, y), F(y, x)) = (x, y)$. The rest of the claim follows
from Proposition \ref{partialorder}.
\end{proof}

We close this section by a theorem that summarizes the conclusion of our analysis.
\begin{theorem}\label{Th-Main}
Let $\Omega$ be a compact subset of $\mathbb{R}^2$ and $z=F(x,y)$ be continuous on $\Omega.$ Suppose that $F(\uparrow,\downarrow).$
If $F$ has a nice extension $\widetilde{F}$ over a rectangular domain containing $\Omega$ such that  $\widetilde{F}$ has no artificial fixed points, then for all  $(x_0,x_{-1})\in \Omega,$ $x_{n+1}=F(x_n,x_{n-1})$ must converge to a fixed point of $F.$
\end{theorem}

\section{Examples}  \label{Sec-Application}
To apply our developed theory on maps of mixed monotonicity, a compact invariant region has to be found, then an extension map $\widetilde{F}$ on a rectangular domain has to be constructed. If $\widetilde{F}$ has no artificial fixed points, i.e., the solution of $(\widetilde{F}(x,y),\widetilde{F}(y,x))=(x,y)$ gives the same fixed points as those of $F,$ then every orbit of $F$ converges to a fixed point of $F.$ To achieve all of this, examples can be involved and lengthy; however, since our purpose here is to give illustrative examples of the developed theory in the previous sections, we give two relatively short examples. The first has an invariant rectangular domain in which the known theory can be used to prove global stability. The second lacks a known invariant rectangular domain, but a compact invariant domain is found. In this case, we apply our developed approach.

\subsection{First example}
Consider the difference equation
\begin{equation}\label{Eq-SecondApplication}
y_{n+1}=F(y_n,y_{n-1})=\frac{p+qy_n}{1+y_n+ry_{n-1}},\quad\text{where}\;\;0<p\leq q \;\;\text{and}\;\; r>0
\end{equation}
and the initial conditions are non-negatives. This problem has been discussed in the literature in which partial stability results have been obtained for specific choices of the parameters \cite{Ku-La2002,Ca-La2008,Am-Ca-La2007,Am-Ca-La2008,Ku-La-Ma-Ro2003}. We are not interested here in writing a survey about the obtained results in the literature, but rather, we are interested in illustrating how to apply the extension-embedding approach here. Note that $F$ is increasing in its first argument and decreasing in its second one within the range of our parameters. Furthermore, $F$ has a unique positive fixed point $\bar y,$ which is less than $q.$ In fact, this positive fixed point is locally asymptotically stable for all values of the parameters.  Now, we show the existence of an invariant compact domain. Since
$$0<F(x,y)\leq \frac{q(1+x)}{1+x+ry}<q,$$
then the orbits of Eq. (\ref{Eq-SecondApplication}) enter the rectangular region $\mathcal{S}:=[0,q]\times [0,q]$ and stay. Thus, $\mathcal{S}$ serves as a compact invariant rectangular domain of $T(x,y)=(F(x,y),x).$ In this case, no need for an extension of $F$ since the embedding theory of Smith \cite{Sm2006,Sm2008} or the global stability results of Kulenovic and Merino \cite{Ku-Me2006} can be applied. Next, we investigate the solution of the system
\begin{equation}\label{Eq-SolveSystem}
F(x,y)=x,\quad F(y,x)=y\quad \text{and}\quad (x,y)\in \mathcal{S}.
\end{equation}
Obviously, $x=y=\bar y$ is a solution, which is a fixed point of $F.$ We need to exclude the existence of artificial fixed points.
\begin{proposition}\label{Pr-Application2}
Consider the parameters $p,q$ and $r$ as given in Eq. (\ref{Eq-SecondApplication}). If one of the following cases is satisfied, then the system in (\ref{Eq-SolveSystem}) has no artificial fixed points in $\mathcal{S}.$

(i)  $q\leq 1$ \hfill (ii) $0\leq r\leq 1$ \hfill (iii) $r>1$ and $p>\frac{1}{4}(r-1)(q-1)^2.$
\end{proposition}
\begin{proof}
(i) Since $y-x=F(y,x)-F(x,y)$ implies $x+y=q-1,$  solutions $(x,y)$ have to be located on the line $x+y=q-1,$ which does not go through the positive quadrant. To verify the second case, we eliminate $y$ from $F(x,y)=x$ and $F(y,x)=y$ to obtain
$$\left[(r-1)x^2-(r-1)(q-1)x+p\right]\left[(-r-1)x^2+(q-1)x+p\right]=0. $$

The factor on the right gives the fixed point $x=y=\bar y.$ The factor on the left gives the artificial fixed points. Thus, we focus on the solutions of $(r-1)x^2-(r-1)(q-1)x+p=0.$ When $r=1,$ we obtain no solution. When $0<r<1,$ we obtain a negative solution and a positive one. Thus, the artificial fixed points are in the second and fourth quadrants. To verify Part (iii), observe the restrictions give non-real roots.
\end{proof}

Now, we apply Theorem \ref{Th-Main} to obtain the following global stability result.

\begin{corollary}
Consider Eq. (\ref{Eq-SecondApplication}) and assume that one of the restrictions in Proposition \ref{Pr-Application2} is satisfied, then the positive equilibrium $(\bar y,\bar y)$ is globally asymptotically stable with respect to the positive quadrant.
\end{corollary}

\subsection{Second example}
Consider the function $F(x,y)=\frac{p+qx}{1+x+ry}-h,\quad q\geq p>h>0$ and $r>0,$  i.e., we consider the difference equation
\begin{equation}\label{Eq-Example2}
x_{n+1}=F(x_n,x_{n-1})=\frac{p+qx_n}{1+x_n+rx_{n-1}}-h,\quad \text{where} \; q\geq p>h>0
\end{equation}
and $ n=0,1,\ldots,x_{-1},x_0\geq 0.$
This equation has  a rich dynamics to the extent that well-defined solutions are not obvious anymore. Since our objective here is to give an illustrative example, we limit our attention to $r=1$ and $q=2p.$ In this case, we have a unique positive equilibrium given by $x^*=p-h.$ The next result characterizes local stability.

\begin{lemma}\label{Eq-LAS}
Consider Eq. (\ref{Eq-Example2}) with $r=1,q=2p$ and $0<h<\min\left\{p,\frac{p+1}{2}\right\}.$ The unique positive equilibrium of Eq. (\ref{Eq-Example2}) is locally asymptotically stable.
\end{lemma}
\begin{proof}
The Jacobian matrix of $F$ at $x^*=p-h$ has trace $T=\frac{p}{1+2x^*}$ and determinant $D=T.$ Thus,  $h<p$ makes $x^*$ positive and $D<1$ makes it locally asymptotically stable.
\end{proof}

When $h=\frac{1}{2},$ F has an infinite number of artificial fixed points, which makes this case a special case that needs different type of treatment. Therefore, we proceed with the assumption $0<h<\frac{1}{2}.$ In the next result, we establish a compact invariant domain.

\begin{lemma}\label{Lem-InvariantDomain}
Consider Eq. (\ref{Eq-Example2}) with $q=2p,r=1$ and $0<h<\min\{p,\frac{1}{2}\}.$ Let $c=\frac{x^*}{h}(x^*+p+1).$ The compact region bounded by the octagon of vertices $A_0:(c,c),$ $A_1:(x^*,c),$ $A_2:(0,x^*),$ $ A_3:(0,0)$ and $A_4:(c,0)$ forms an invariant domain.
\end{lemma}
\begin{proof}
Let $\Omega$ be the compact region bounded by the boundary curves $\gamma_j(t)=(1-t)A_j+tA_{j+1 \bmod 5},\; j=0,\ldots, 4$ and $0\leq t\leq 1.$ We show that $\Omega$ is invariant under $T(x,y)=(F(x,y),x).$ Since $T$ is one-to-one on the positive quadrant, it is sufficient to show that $T(\gamma_j) \subset \Omega.$  Observe that $x^*<\frac{x^*}{h}<c.$ The rest of the proof is computational, and we give the main outlines here. We have $ T(\gamma_0)=T(c+t(x^*-c),c)$ with $T(\gamma_0(0))=A_1$ and $T(\gamma_0(1))=A_2.$ Since also the curve is convex, we obtain $T(\gamma_0)\subset\Omega.$
$T(\gamma_1)=T((1-t)x^*,(1-t)c+tx^*),$ which is increasing in $x,$ decreasing in $y,$ and satisfies $T(\gamma_1(0))=A_2,$ $T(\gamma_1(1))=\left(\frac{p}{x^*+1}-h,0\right).$ Thus, $T(\gamma_1) \subset \Omega.$  $T(\gamma_2)=\left(\frac{p}{1+(1-t)x^*}-h,0\right)\subset \gamma_3.$
$$T(\gamma_3)=T(tc,0)=\left(\frac{p(1+2ct)}{1+ct}-h,ct\right).$$
The second component increases from $0$ to $c,$ while the first component increases from $x^*$ to $x^*+\frac{pc}{1+c}.$ It remains to show that
 $x^*+\frac{pc}{1+c}<c.$ We have
 $$ (c-x^*)(1+c)-pc>0 \quad \Leftrightarrow\quad \frac{x^*}{h^2}\left(4p^2(p-3h+1)+(1-3h)^2p+h^2(1-2h)\right)>0,$$
 which holds true when $0<h<\{p,\frac{1}{2}\}.$
  Finally $T(\gamma_4)\subset \gamma_0.$
Hence, the proof is complete (Fig. \ref{Fig-CompactInvariant} gives an illustration).
\end{proof}
Next, we consider the invariant region $\Omega$ as the domain of $F,$ and construct a nice extension $\widetilde{F}$ on the rectangular region $\mathcal{S}=[0,c]^2.$ Since $F(\gamma_1(t))$ is increasing when $0<t<1$ and $0<h<\frac{1}{2},$  our nice extension will be
$$\widetilde F(x,y) = \begin{cases} F(x,y),&\text{if}\; (x,y)\in \Omega\\
F(x^+,y),& \text{if}\; (x,y)\in \mathcal{S}\setminus \Omega,
\end{cases}$$
where $x^+=\frac{(y-x^*)x^*}{(c-x^*)}.$ Next, we show that $\widetilde{F}$ has no artificial fixed points.

\begin{center}
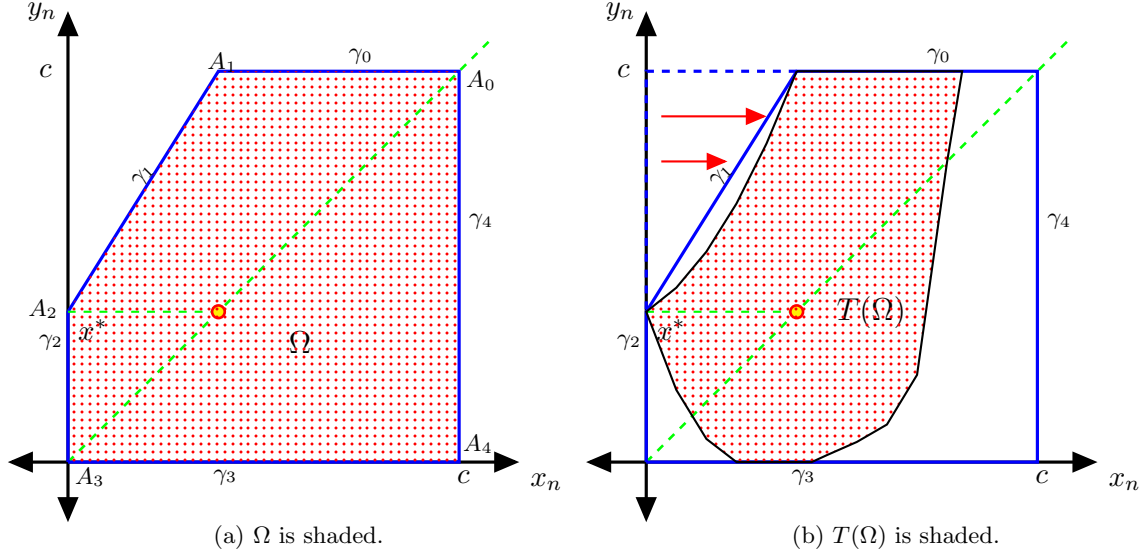
\begin{figure}[!h]
\definecolor{ffqqqq}{rgb}{1.,0.,0.}
\definecolor{qqqqff}{rgb}{0.,0.,1.}
\begin{minipage}[t]{0.48\textwidth}
\raggedleft
\begin{center}
\begin{tikzpicture}[scale=0.40]
\draw[-triangle 45, line width=1.2pt,scale=1.0] (0,0) -- (15.0,0) node[below right] {$x_{n}$};
\draw[-triangle 45, line width=1.2pt,scale=1.0] (0,0) -- (0,15.0) node[left] {$y_n$};
\draw[line width=1.2pt,-triangle 45] (0,0) -- (-2.0,0);
\draw[line width=1.2pt,-triangle 45] (0,0) -- (0.0,-2.0);
\draw[line width=1.0pt,xscale=1.0,domain=0.0:14.0,smooth,variable=\x,green,dashed] plot ({\x},{\x});
\draw[line width=1.2pt,blue] (0,0)--(13,0)--(13,13)--(5,13)--(0,5)--(0,0);
\draw[line width=0.8pt,dashed,green] (5,5) -- (0,5);
\draw[line width=1.0pt,red,fill=yellow] (5,5) circle (6pt);
\draw[scale=1.0] (12.6,-0.5) node[right,rotate=0] { $c$};
\draw[scale=1.0] (-1.3,13.0) node[right,rotate=0] { $c$};
\draw[scale=1.0] (0.0,4.4) node[right,rotate=0] { $x^*$};
\draw[scale=1.0] (4.0,8.0) node[right,rotate=45] {\large $\mathcal{\Omega}$};
\draw[scale=1.0] (-1.3,4.0) node[right,rotate=0] {\footnotesize $\gamma_2$};
\draw[scale=1.0] (4.5,-0.5) node[right,rotate=0] {\footnotesize $\gamma_3$};
\draw[scale=1.0] (2.0,9.0) node[right,rotate=45] {\footnotesize $\gamma_1$};
\draw[scale=1.0] (13.0,8.0) node[right,rotate=0] {\footnotesize $\gamma_4$};
\draw[scale=1.0] (9.0,13.5) node[right,rotate=0] {\footnotesize$\gamma_0$};
\draw[scale=1.0] (7.0,4.0) node[right,rotate=0] {\large $\Omega$};
\draw[scale=1.0] (12.9,12.7) node[right,rotate=0] {\footnotesize $A_0$};
\draw[scale=1.0] (4.3,13.3) node[right,rotate=0] {\footnotesize $A_1$};
\draw[scale=1.0] (0.0,5.1) node[left,rotate=0] {\footnotesize $A_2$};
\draw[scale=1.0] (-0.1,-0.5) node[right,rotate=0] {\footnotesize $A_3$};
\draw[scale=1.0] (12.8,0.6) node[right,rotate=0] {\footnotesize $A_4$};
\fill[line width=0.pt,color=ffqqqq,fill=ffqqqq,pattern=dots,pattern color=ffqqqq] (0,0)--(13,0)--(13,13)--(5,13)--(0,5)--(0,0)-- cycle;
\draw[scale=1.0] (4.5,-2.5) node[right] {\footnotesize (a) $\Omega$ is shaded.};
\end{tikzpicture}
\end{center}
\end{minipage}
\begin{minipage}[t]{0.48\textwidth}
\raggedright
\begin{center}
\begin{tikzpicture}[scale=0.40]
\draw[-triangle 45, line width=1.2pt,scale=1.0] (0,0) -- (15.0,0) node[below right] {$x_{n}$};
\draw[-triangle 45, line width=1.2pt,scale=1.0] (0,0) -- (0,15.0) node[left] {$y_n$};
\draw[line width=1.2pt,-triangle 45] (0,0) -- (-2.0,0);
\draw[line width=1.2pt,-triangle 45] (0,0) -- (0.0,-2.0);
\draw[line width=1.0pt,xscale=1.0,domain=0.0:14.0,smooth,variable=\x,green,dashed] plot ({\x},{\x});
\draw[line width=1.2pt,blue] (0,0)--(13,0)--(13,13)--(5,13)--(0,5)--(0,0);
\draw[line width=0.8pt,dashed,green] (5,5) -- (0,5);
\draw[line width=1.0pt,red,fill=yellow] (5,5) circle (6pt);
\draw[scale=1.0] (12.6,-0.5) node[right,rotate=0] { $c$};
\draw[scale=1.0] (-1.3,13.0) node[right,rotate=0] { $c$};
\draw[scale=1.0] (0.0,4.4) node[right,rotate=0] { $x^*$};
\draw[scale=1.0] (6.0,5.0) node[right,rotate=0] {\large $T(\Omega)$};
\draw[scale=1.0] (-1.3,4.0) node[right,rotate=0] {\footnotesize $\gamma_2$};
\draw[scale=1.0] (4.5,-0.5) node[right,rotate=0] {\footnotesize $\gamma_3$};
\draw[scale=1.0] (2.0,9.0) node[right,rotate=45] {\footnotesize $\gamma_1$};
\draw[scale=1.0] (13.0,8.0) node[right,rotate=0] {\footnotesize $\gamma_4$};
\draw[scale=1.0] (9.0,13.5) node[right,rotate=0] {\footnotesize$\gamma_0$};
\draw[-triangle 45, line width=0.8pt,red] (0.5,11.5) -- (4.0,11.5);
\draw[-triangle 45, line width=0.8pt,red] (0.5,10) -- (2.7,10);
\draw[line width=1.2pt,blue,dashed] (5,13)--(0,13)--(0,5);
\draw[line width=0.8pt,color=black,pattern=dots,pattern color=ffqqqq]  plot[smooth, tension=.0] coordinates {(5,13) (4.0,10.6) (3,8.6) (2,7) (1,5.8) (0,5) (1,2.4) (2.0,0.77) (3,0) (3.5,0) (4,0) (4.5,0) (5.5,0) (6,0.23) (7,0.67) (8,1.25) (9,2.9) (10,10) (10.5,13) (10,13) (8,13) (6,13) (5,13)};
\draw[scale=1.0] (4.5,-2.5) node[right] {\footnotesize (b) $T(\Omega)$ is shaded.};
\end{tikzpicture}
\end{center}
\end{minipage}%
\caption{Part (a) of this figure shows the invariant region $\Omega,$ while part (b) shows  $T(\Omega)$ together with the extension through horizontal projections. The scale is missing to indicate the general form of the region when $0<h<\min\{p,\frac{1}{2}\}.$  }\label{Fig-CompactInvariant}
\end{figure}
\end{center}

\begin{proposition}\label{Pr-NoArtificialFixedPoints}
The map $\widetilde{F}$  has no artificial fixed points in $\mathcal{S}=[0,c]^2.$
\end{proposition}

\begin{proof}
For $(x,y)\in \Omega,$ $\widetilde{F}(x,y)=F(x,y),$ and it is straightforward to find that the only solution of $(F(x,y),F(y,x))=(x,y)$ is $(x^*,x^*)$ as long as $0<h<\frac{1}{2}.$ Next, consider $(x,y)\in \mathcal{S}\setminus \Omega.$ This means we have to seek solutions $(x,y)$ that satisfy
\begin{equation}\label{Eq-Artificial1}
\widetilde{F}(x,y)=F(x^+,y)=F\left(\frac{(y-x^*)x^*}{c-x^*},y\right)=x\quad\text{and}\quad F(y,x)=y,
\end{equation}
 where $ y>\frac{c-x^*}{x^*}x+x^*$ and $ 0<x<x^*.$
 Handling the two nonlinear equations together with the inequality is a cumbersome task, but we get around it by considering equations (\ref{Eq-Artificial1}) together with $y=mx+x^*,$ and show that no solution exist for $m>\frac{c-x^*}{x^*}=\frac{2x^*+1}{h}.$  This covers the points in $\mathcal{S}\setminus \Omega$ except the $y$-axis. Thus, we investigate solutions of
 $$F\left(\frac{(mx)x^*}{c-x^*},mx+x^*\right)=x\qquad\text{and}\quad  F(mx+x^*,x)=mx+x^*.$$
 For computational conveniences, substitute $p=x^*+h$ and eliminate $x$ from the two equations to obtain a cubic equation in $m$ of the form
 $$a_3(x^*,h)m^3+a_2(x^*,h)m^2+a_1(x^*,h)m+a_0(x^*,h)=0.$$
  Replace $m$ by $M+\frac{c-x^*}{x^*}$ to obtain
  \begin{equation} \label{Eq-Mvalues}
  b_3(x^*,h)M^3+b_2(x^*,h)M^2+b_1(x^*,h)M+b_0(x^*,h)=0.
  \end{equation}
  The coefficients $b_j$ have same sign, which can be verified by writing each $b_j$ as a polynomial in $x^*,$ then depend on the assumption that $0<h<\frac{1}{2}.$ We show $b_3,$ which is
$$ b_3(x^*,h)=h^3(1-h)^2(4(x^*)^3+4(x^*)^2+(1-h)(3h+1)x^*+h(1-h-h^2).$$
The other coefficients are handled similarly.
To this end, it has been shown that no $M>0$ satisfies equation (\ref{Eq-Mvalues}), and consequently no $m>\frac{c-x^*}{x^*}.$
Finally, it remains to consider solutions of $\widetilde{F}(x,y)=x$ and $F(y,x)=y$ when $x=0,$ which is simple and has no possible solutions. Hence, the proof is complete.
\end{proof}
Now, we have all the needed machinery to prove global stability with respect to $\Omega.$

\begin{corollary}
Consider Eq. (\ref{Eq-Example2}) with $0<h<\min\{p,\frac{1}{2}\}.$ The positive equilibrium is globally asymptotically stable with respect to the invariant set $\Omega.$
\end{corollary}
\begin{proof}
Proposition \ref{Pr-NoArtificialFixedPoints} shows that we have no artificial fixed points for $\widetilde{F}$ as long as $0<h<\min\{p,\frac{1}{2}\}.$ Now, Theorem \ref{Th-Main}  gives the global stability for the extension $\widetilde{F}$ on $\mathcal{S},$ which gives the global stability for   $F$ on $\Omega.$
\end{proof}
It is worth mentioning that it is possible to prove the global stability with respect to the persistent set rather than $\Omega,$ but it is a matter of computations to show that orbits of $T$ get inside $\Omega$ in few iterations.

\section{Conclusion and discussion}
This paper is concerned with  $2$-dimensional continuous maps  of mixed monotonicity that are defined on  compact domains. When the compact domain is convex or semi-convex and the boundary is a piecewise smooth Jordan curve, a continuous function $F$ can be extended to a function $\widetilde{F}$ on a rectangular domain, having the same monotonicity properties, and the extension can be chosen so that its image coincides with that of $F.$ In this case, we call it \emph{nice extension.} We believe this result is generalizable to non-convex domains under proper conditions on the boundary; however,  the general proof is technical. It can be developed further based on the need in real applications.
\\

The embedding technique takes the following convenient form :
Let $\Omega$ be a compact subset of $\mathbb{R}^2$ and $z=F(x,y)$ be continuous on $\Omega.$ Suppose that $F(\uparrow,\downarrow).$
If $F$ has a nice extension $\widetilde{F}$ over a rectangular domain containing $\Omega$ such that  $\widetilde{F}$ has no artificial fixed points, then $x_{n+1}=F(x_n,x_{n-1})$ must converge to a fixed point of $F$ for all  $(x_0,x_{-1})\in \Omega.$ The concept of \emph{artificial fixed points} is used to denote solutions of the system
$$(\widetilde{F}(x,y),\widetilde{F}(y,x))=(x,y)$$
that are not fixed points of $\widetilde{F}.$ The use of this concept is driven by the fact that solutions of this system are in fact fixed points of the higher dimensional map $G$ used in our embedding (see Corollary \ref{Th-Ahmad2} or Corollary \ref{Th-Ahmad}).
This result can be considered as a high dimensional generalization of Coppel's result \cite{Co1955}, which states the following: Suppose $f:\;[a,b]\rightarrow\;[a,b]$ is continuous. A necessary and sufficient condition that the iteration sequence $x_{n+1}=f(x_n)$ converges, whatever initial point $x_0\in [a,b]$ is chosen, is that the equation $f(f(x))=x$ has no roots except the roots of the equation $f(x)=x.$
\\

Two examples have been considered. One of them shows the existence of an invariant box in which global stability has been obtained based on the available results in the literature without the need for an extension. The second example shows the existence of a compact invariant domain that is not a box. In this case, we constructed a nice extension and used our approach to obtain global stability with respect to the compact invariant domain.
\\

Finally, it is worth mentioning that the main obstacle facing our approach is the solution of the system $(\widetilde{F}(x,y),\widetilde{F}(y,x))=(x,y).$ Although the solution of the system can be unique with respect to the original map $F,$ it may not be unique with respect to the extended map $\widetilde{F}.$ From our perspective, this challenge can be tackled by seeking an alternative invariant domain, or by investigating the basin of attraction of each individual equilibrium point, i.e., the original equilibrium points as well as the artificial ones. However, this subject needs further research before a definite  answer can be given.
\\


\bibliographystyle{plain}

\vskip 10pt
{\small
\noindent{\sc Address first author:}\\ Sultan Qaboos University, P. O. Box 36,
 PC 123, Al-Khod, Sultanate of Oman.\\

\noindent{\sc Address second and third authors:}\\ Department of Mathematics and Statistics,
 American University of Sharjah, P. O. Box 26666, University City, Sharjah, UAE}

\end{document}